\documentclass{amsproc}
\usepackage{euscript}
\usepackage{cases}
\usepackage{mathrsfs}
\usepackage{bbm}
\usepackage{amssymb}
\usepackage{amsfonts,amsmath,amsxtra,mathdots,mathabx,pifont}
\usepackage{color}
\usepackage{hyperref}
\usepackage{tikz}
\usepackage{appendix,upgreek}

\allowdisplaybreaks

\DeclareFontFamily{U}{matha}{\hyphenchar\font45}
\DeclareFontShape{U}{matha}{m}{n}{
	<5> <6> <7> <8> <9> <10> gen * matha
	<10.95> matha10 <12> <14.4> <17.28> <20.74> <24.88> matha12
}{}
\DeclareSymbolFont{matha}{U}{matha}{m}{n}

\DeclareMathSymbol{\Lt}{3}{matha}{"CE}
\DeclareMathSymbol{\Gt}{3}{matha}{"CF}

\DeclareSymbolFont{mathc}{OML}{txmi}{m}{it}
\DeclareMathSymbol{\varuu}{\mathord}{mathc}{117}
\DeclareMathSymbol{\varvv}{\mathord}{mathc}{118}
\DeclareMathSymbol{\varww}{\mathord}{mathc}{119}

\def\lp{\left(}
\def\rp{\right)}

\def\valpha{\text{\scalebox{0.88}[1.02]{$\alpha$}}}   
\def\vepsilon{\upvarepsilon}

\def\vchi{\text{\raisebox{0.6 \depth}{\scalebox{0.9}[1.1]{$\chi$}}}} 
\def\vlambda{\text{\scalebox{0.9}[1]{$\lambda$}}}

\def\vnu{\text{{\scalebox{0.86}[1]{$\nu$}}}} 
\def\varnu{\text{{\scalebox{0.9}[1]{$\upnu$}}}}
\def\vvarkappa{\text{{\scalebox{0.9}[1]{$\upkappa$}}}}
\def\uppii{\text{\scalebox{0.8}[0.96]{$\uppi$}}}

\def\bfnu{\boldsymbol{\text{{\scalebox{0.86}[1]{$\upnu$}}}}}
\def\bfkappa{\boldsymbol{\text{{\scalebox{0.86}[1]{$\upkappa$}}}}}

\def\bfn{\boldsymbol{\text{{\scalebox{0.9}[1]{$n$}}}}}
\def\bfm{\boldsymbol{\text{{\scalebox{0.9}[1]{$m$}}}}}
\def\bfu{\boldsymbol{\text{{\scalebox{0.9}[1]{$u$}}}}}

\def\bfN{\boldsymbol{\text{{\scalebox{0.9}[1]{$N$}}}}}

\newcommand{\BC}{{\mathbb {C}}}
\newcommand{\BH}{{\mathbb {H}}} 
\newcommand{\BQ}{{\mathbb {Q}}} 
\newcommand{\BR}{{\mathbb {R}}} 
\newcommand{\BZ}{{\mathbb {Z}}}

\def\RN{\mathrm{N}}

\newcommand{\SL}{{\mathrm {SL}}}

\newcommand{\ra}{\rightarrow} 
\def\viint{	\int \hskip -5 pt \int} 

\def\sumo{\sideset{}{^{ \mathrm{o} }}\sum}
\def\sumf{\sideset{}{^\flat}\sum}

\def\nd{\mathrm{d}}
\def\trh{ \mathrm{trh}}
\def\Tr{ \mathrm{Tr}}

\def\SB{\text{\raisebox{- 2 \depth}{\scalebox{1.1}{$ \text{\usefont{U}{BOONDOX-calo}{m}{n}B} \hskip 0.5pt $}}}}

\def\SO {\text{\raisebox{- 2 \depth}{\scalebox{1.1}{$ \text{\usefont{U}{BOONDOX-calo}{m}{n}O}  $}}}}

\def\SDH{\text{\raisebox{- 6 \depth}{\scalebox{1.06}{$ \text{\usefont{U}{dutchcal}{m}{n}H}  $}}}}
\def\SDI{\text{\raisebox{- 6 \depth}{\scalebox{1.06}{$ \text{\usefont{U}{dutchcal}{m}{n} I}  $}}}}

\def\shskip{\hskip 0.5 pt}

\def\frc{\mathfrak{c}}
\def\frd{\mathfrak{d}}

\def\frm{\mathfrak{m}}
\def\frn{\mathfrak{n}}
\def\frp{\mathfrak{p}}
\def\frq{\mathfrak{q}}

\newcommand{\delete}[1]{}

\theoremstyle{plain}

\newtheorem{thm}{Theorem} \newtheorem{cor}{Corollary}[section]
\newtheorem{lem}{Lemma}[section]
\newtheorem{defn}{Definition}[section] \newtheorem{prop}[thm]{Proposition}

\theoremstyle{remark} 
\newtheorem{remark}{Remark}[section]

\numberwithin{equation}{section}

\begin{document}

	\title[Mean Lindel\"of Hypothesis and Prime Geodesic Theorem]{Mean Lindel\"of Hypothesis and Prime Geodesic Theorem on Bianchi Manifolds}

\begin{abstract}
	 In this paper, we prove the mean Lindel\"of hypothesis for the second moment of symmetric square $L$-functions and thereby obtain the (unconditional) error bound $O (x^{36/23+\vepsilon})$ in the prime geodesic theorem on Bianchi manifolds. 
\end{abstract}

	\author[C. Li and Z. Qi]{Changlin Li  and Zhi Qi}
\address{School of Mathematical Sciences\\ Zhejiang University\\Hangzhou, 310027\\China}
\email{12135012@zju.edu.cn, zhi.qi@zju.edu.cn} 
\thanks{The second author was supported by National Key R\&D Program of China No. 2022YFA1005300 and National Natural Science Foundation of China No. 12071420.}


\maketitle

\section{Introduction}

Let  $F $ be an imaginary quadratic field and $ \SO$ be its ring of integers. Assume that the class number $h_F = 1$. Let the Bianchi group $\varGamma = \SL_2 (\SO)$ act on the hyperbolic space $\BH^3$.  

In 1983, the prime geodesic theorem on the Bianchi manifold $\varGamma \backslash \mathbb{H}^3$ was established by Sarnak in the seminal work \cite{Sarnak-H3}:
\begin{align}\label{1eq: PGT, 2}
	\pi_{\varGamma}  (x) = \mathrm{li} (x^2) + E_{\varGamma}  (x), 
\end{align}
and he proved the benchmark error bound $ E_{\varGamma} (x) = O (x^{5/3+\vepsilon}) $, where  $\pi_{\varGamma} (x) $   counts the prime geodesics on $\varGamma \backslash \mathbb{H}^3$ of norm below $ x$ and  $ \mathrm{li} (x) $ is  the logarithmic integral as in the prime number theorem.

Let    $ \SB  $ be an orthonormal basis of Hecke--Maass cusp forms in $ L^2 (\varGamma \backslash \mathbb{H}^3) $. For $f \in \SB$ let $1+ 4 t_f^2$ be its Laplace eigenvalue and $L (s, \mathrm{Sym}^2 f)$ be its symmetric square $L$-function.  

In 2001, for $F = \BQ (i)$ (so that $\varGamma \backslash \mathbb{H}^3$ is the Picard manifold) Koyama  \cite{Koyama-PGT-Picard} proved 
\begin{align}\label{1eq: bound Koyama}
	E_{\varGamma} (x) \Lt x^{11/7+\vepsilon}  , 
\end{align}  conditionally on the mean Lindel\"of hypothesis: 
\begin{align}\label{1eq: sym2 mean, 1}
	\sum_{T \leqslant  t_f  \leqslant 2 T}  | L (s, \mathrm{Sym}^2 f)|^2 \Lt |s|^{A} T^{3+\vepsilon} , \qquad \mathrm{Re} (s) = \frac 1 2 .
\end{align}  
The mean Lindel\"of hypothesis on the modular surface $\SL_2 (\BZ) \backslash \BH^2$ was conjectured (for the Rankin--Selberg $L (s, f \times f) = \zeta (s) L (s, \mathrm{Sym}^2 f)$) in 1984 by Iwaniec \cite{Iwaniec-PGT} and settled in 1995 by Luo and Sarnak \cite{Luo-Sarnak-QE}. However, on the Picard manifold $ \SL_2 (\BZ[i]) \backslash \BH^3 $ this conjecture has been a long-standing open problem. Note that the  convexity bound for $ L (s, \mathrm{Sym}^2 f)$ yields trivially $ O (|s|^{3+\vepsilon} T^{5+\vepsilon})$.  It is only recent that non-trivial bounds  were obtained by different methods in \cite{BCCFL-PGT-Picard-1}, \cite{BF-Picard-Sym2}, and \cite{Qi-Sym2-LS}:  $$  O \big(|s|^{A} T^{4+\vepsilon} \big) , \qquad  O   \big(|s|^{A} T^{11/3+\vepsilon}\big), \qquad  O   \big(|s|^{3+\vepsilon} T^{7/2+\vepsilon}\big) , $$ 
and every one of them leads to an improved (unconditional)  bound for $E_{\varGamma} (x)$. 

The second major ingredient in the study of prime geodesic theorem is  the subconvexity of  quadratic Dirichlet $L$-functions  (over $\SO$), that is,
\begin{align}\label{1eq: subconvex bound}
	L (s, \vchi_{\mathfrak{q}} ) \Lt |s|^A \RN (\mathfrak{q})^{\theta +\vepsilon}, \qquad \mathrm{Re} (s) = \frac 1 2, 
\end{align}
where $ \vchi_{\mathfrak{q}} $ is the quadratic Dirichlet character of conductor $\mathfrak{q}$. The currently best Weyl subconvexity exponent $\theta = 1/6$ is due to Nelson \cite{Nelson-Eisenstein}---it is a number-field generalization of the work of Conrey and Iwaniec \cite{CI-Cubic}. 

In 2022, Balog et al. \cite{BBCL-PGT-Picard-3} improved Koyama's conditional error bound \eqref{1eq: bound Koyama} into 
\begin{align}\label{1eq: error bound}
	E_{\varGamma} (x) \Lt x^{  3 / 2 +   {(24 \theta -1 ) } / {46 } + \vepsilon  } , 
\end{align}
or 
\begin{align}\label{1eq: error bound, 2}
	E_{\varGamma} (x) \Lt x^{    {36  } / {23 } + \vepsilon  } , 
\end{align}
on the Weyl exponent $\theta = 1/6$. This is the analogue of the currently best error bound $E_{\varGamma} (x) = O (x^{25/36+\vepsilon})$ in the case $\varGamma = \SL_2 (\BZ)$ by Soundararajan and Young \cite{Sound-Young-PGT}.  

In view of Remark 4.2 in \cite{Koyama-PGT-Picard}, the Kuznetsov formula in \cite{BM-Kuz-Spherical,Venkatesh-BeyondEndoscopy,B-Mo2}, and the Kuznetsov--Bykovski\u{\i} formula in \cite{C-Wu-Z-KB-Formula}, the error bounds in \eqref{1eq: bound Koyama}, \eqref{1eq: error bound}, and \eqref{1eq: error bound, 2} should be valid for Bianchi manifolds as in the original setting of Sarnak \cite{Sarnak-H3}.  

In this paper, on Bianchi manifolds we prove the mean Lindel\"of hypothesis \eqref{1eq: sym2 mean, 1} and thereby make the error bound \eqref{1eq: error bound, 2} unconditional. 

\begin{thm}\label{thm: mean-Lindelof} Let the notation be as above. Let $\vepsilon > 0$ and $T \geqslant 1$.  
	Then 
	\begin{align}\label{1eq: sym2 mean, 2}
		\sum_{T \leqslant t_f  \leqslant 2T  }  | L (s, \mathrm{Sym}^2 f)|^2 \Lt_{\vepsilon}  |s|^{{12}} T^{3+\vepsilon} , 
	\end{align}  
 for $\mathrm{Re} (s)  = 1/2$, where the    implied constant depends only on $\vepsilon$. 
\end{thm}

\begin{remark}
	Note that in the case $|s| \geqslant T$ the bound in \eqref{1eq: sym2 mean, 2} follows trivially from the convexity bound $L (s, \mathrm{Sym}^2 f) \Lt |s|^{3/2+\vepsilon}$, so subsequently we shall assume 
	\begin{align}\label{1eq: s<T}
		|s| <  T. 
	\end{align}
\end{remark}

\begin{thm}\label{thm: PGT}
	Let the notation be as above. We have
	\begin{align}
			\pi_{\varGamma}  (x) = \mathrm{li} (x^2) +  O \big(x^{36/23+\vepsilon} \big), 
	\end{align} 
for any $\vepsilon > 0$. 
\end{thm}

The prime geodesic theorem has been intensely studied from various perspectives by many mathematicians. The reader is referred to \cite{Iwaniec-PGT,Luo-Sarnak-QE,Bykovskii,Cai-PGT,Sound-Young-PGT} for that on the modular surface $\SL_2 (\BZ) \backslash \BH^2$ and to   \cite{Koyama-PGT-Picard,BCCFL-PGT-Picard-1,BF-PGT-Picard-2,BF-Picard-Sym2,Qi-Sym2-LS} for that on the Picard manifold $\SL_2 (\BZ[i]) \backslash \BH^3$; the list of references here is by no means complete.

Our primary focus of course will be on the proof of Theorem \ref{thm: mean-Lindelof}. 
Our idea is to use the standard Kuznetsov/Vorono\"i approach along with the quadratic large sieve of Heath-Brown as in the work of Khan and Young \cite{Khan-Young-Sym2}.  

For $T^{1/5} \leqslant M \leqslant T^{1-\vepsilon}$, it follows from their Theorem 1.3\footnote{More explicitly, Khan and Young proved for $T^{1/5} \leqslant M \leqslant T^{1-\vepsilon}$ and $|s| \leqslant T^{\vepsilon}$ that \begin{align*}
		\sum_{T \leqslant t_f  \leqslant T+ M  }  | L (s, \mathrm{Sym}^2 f)|^2 \Lt_{\vepsilon} M T^{1+\vepsilon} , 
\end{align*} 
so the exponent $A$ in \eqref{1eq: sym2 mean, K-Y} depends implicitly on $\vepsilon$. } that the mean Lindel\"of hypothesis holds:  
\begin{align}\label{1eq: sym2 mean, K-Y}
	\sum_{T \leqslant t_f  \leqslant T+ M  }  | L (s, \mathrm{Sym}^2 f)|^2 \Lt |s|^{A} M T^{1+\vepsilon} , 
\end{align} 
for $\mathrm{Re} (s) = 1/2$, where $ f $ traverses an orthonormal basis of Maass cusp forms on $ \SL_2 (\BZ) \backslash \BH^2 $. For other bounds or asymptotics for the moments of symmetric square $L$-functions for $\BH^2$, we refer the reader to 
\cite{Balkanova-Sym2-Maass,Blomer-Sym2-Hol,BF-Mean-Sym2,BF-Sym2-Non-vanishing,BF-KY-Sym2-2,Khan-Das-3rd-Moment,BF-KY-Sym2-1,Iwaniec-Michel-Sym2,Khan-Sym2-Non-vanishing,Lam-Sym2-Hol,Liu-Sym2-1st-Moment,Luo-Sym2-2nd-Moment,Nelson-Sym2,Tang-Xu-Sym2-Maass}. In particular, Frolenkov \cite{BF-KY-Sym2-1} recently provided a quite different approach to \cite[Theorem 1.3]{Khan-Young-Sym2} via the Vorono\"i summation formula for Zagier $L$-functions \cite{BF-Voronoi}. 

\begin{prop}\label{prop: main}
Let the setting be on $ \SL_2 (\SO) \backslash \BH^3$.	 Let $M, T$ be large parameters with $ T^{\vepsilon} \leqslant M \leqslant T^{1-\vepsilon} $. 
	 Then for $\mathrm{Re} (s) = 1/2$ we have
	 \begin{align}\label{1eq: sym2 mean, 3}
	 	\sum_{T \leqslant t_f  \leqslant T + M }  | L (s, \mathrm{Sym}^2 f)|^2 \Lt_{\vepsilon}  |s|^{ {12}}  T^{3+\vepsilon} . 
	 \end{align}  
\end{prop}

\begin{remark}
	Actually, instead of {\rm\eqref{1eq: sym2 mean, 3}}, we shall prove 
	\begin{align}\label{1eq: sym2 mean, 4}
		\sum_{T \leqslant t_f  \leqslant T + M }  | L (s, \mathrm{Sym}^2 f)|^2 \Lt_{\vepsilon}  \frac {|s|^{ {12}}  T^{4+\vepsilon}} {M} .  
	\end{align}  

\end{remark}

Clearly    
\eqref{1eq: sym2 mean, 2}, \eqref{1eq: sym2 mean, 3}, and \eqref{1eq: sym2 mean, 4} are equivalent to one another as  the latter two are both mean Lindel\"of  for (but only for) $M = T^{1-\vepsilon}$. The introduction of $M$ in Proposition \ref{prop: main} seems artificial but it will help us facilitate the proof by applying the analysis in \cite{Qi-GL(3)} (however, the focus therein is the case $M = T^{\vepsilon}$). 

Note that the result of Khan and Young in \eqref{1eq: sym2 mean, K-Y} is valid for $ M = T^{1/5}$. The reader must wonder why one is not able to prove for some $M = T^{1-\delta}$ (with $\delta > 0$ fixed)  in   \eqref{1eq: sym2 mean, 3} or \eqref{1eq: sym2 mean, 4} the similar mean Lindel\"of bound $O_{s} (  MT^{2+\vepsilon})$. The simple answer is that---as remarked in \S 1.4 of \cite{Qi-Sym2-LS}---it should be   those  non-spherical cuspidal representations $\uppii$ on $ \SL_2 (\SO) \backslash \SL_2 (\BC)  $ with spectral parameters $ T \leqslant |t_{\uppii}| \leqslant T+M $ and $T \leqslant |p_{\uppii} | \leqslant T + M $ that are truly analogous to those Maass cusp forms $f$  on $ \SL_2 (\BZ) \backslash \BH^2 $ with $  T \leqslant t_f \leqslant T+M $.  However, if we restrict our attention to $ \SL_2 (\SO) \backslash \BH^3  $ and  $ \SL_2 (\BZ) \backslash \BH^2 $, then  the  (analytic) difficulty of the present problem on $   \SL_2 (\SO) \backslash \mathbb{H}^3 $ is caused by the asymmetry exhibited in the (domain of)  complex Bessel integral:
\begin{align*}
	M T^2 \int_{0}^{\pi} \int_{-M^{\vepsilon}/M}^{M^{\vepsilon}/M}   g (M r) e (  2 Tr/\pi \pm  4 \mathrm{Re} (z  \trh (r, \omega)  ))  \nd r \nd \omega, 
\end{align*}
with
\begin{align*}
	 \trh (r, \omega) = \cosh r \cos \omega + i \sinh r \sin \omega, 
\end{align*}
in comparison with the real Bessel integral:
\begin{align*}
	M T \int_{-M^{\vepsilon}/M}^{M^{\vepsilon}/M}   g (M r) e (Tr /\pi \pm 2 x \cosh r ) \nd r .
\end{align*}
See \cite[Lemma 8.2]{Qi-GL(3)}.  
Consequently, in the most typical cases, 
\begin{itemize}
	\item [(1)] the Poisson summation is less powerful in the setting of $\SL_2 (\SO) \backslash \mathbb{H}^3$  than  $\SL_2 (\BZ) \backslash \mathbb{H}^2$:  one could {\it not} reduce length of summations by Poisson; 
	
	\item [(2)] the stationary-phase estimate (after Poisson and Mellin) is less beneficial in the setting of $\BC$ than $\BR$: one could {\it not} save more than trivial estimation.  
\end{itemize} 
Therefore, our proof will be different from \cite{Khan-Young-Sym2} in that the third Poisson will  {not} be applied, and our analysis will be simpler than \cite{Qi-GL(3)} in that the final integral will be bounded trivially.  Our main saving is indeed of arithmetic nature and the  double Poisson only serves as a bridge to the quadratic large sieve of Heath-Brown.  

For more detailed discussions on (1) and (2), the reader is referred to Remarks \ref{rem: c-length}, \ref{rem: worse than trivial}, and \ref{rem: 3rd Poisson}.

\subsection*{Notation}

By $X \Lt Y$ or $X = O (Y)$ we mean that $|X| \leqslant c Y$  for some constant $c  > 0$, and by $X \asymp Y$ we mean that $X \Lt Y$ and $Y \Lt X$. We write $X \Lt_{\valpha, \beta, ...} Y $ or $  X = O_{\valpha, \beta, ...} (Y) $ if the implied constant $c$ depends on $\valpha$, $\beta$, ....  

The notation $x \sim X$ stands for  $ X <  x \leqslant 2 X $.  

By `negligibly small' we mean $ O_A ( T^{-A} )$ for arbitrarily large but fixed $A > 0$. 

Throughout the paper,  $\vepsilon  $ is arbitrarily small and its value  may differ from one occurrence to another.

\section{Preliminaries}

\subsection{Basic Notation}

Let $F = \BQ ( \sqrt{d_{F}} )$ be imaginary quadratic (of  discriminant $d_{F} < 0$). Assume that the class number $h_{F} = 1$. It is well known that
\begin{align*}
	d_F = -3, -4, -7, -8, -11, -19, -43, -67, -163. 
\end{align*} Let $\SO$ be its ring of integers and $\SO ^{\times}$ be its group of units. Write  $\updelta_F = \sqrt{d_F}$. Note that 
$ \SO' = \updelta_F^{-1} \SO  $ is the dual of $\SO$. Denote $w_F = |\SO^{\times}|$. 
Let $\mathrm{N}$ and $\Tr$ denote the norm and the trace for $F \subset \BC$, respectively. 

For a non-zero ideal $\frn = (n)$, we usually choose the standard $n$ so that $\arg (n)  \in \allowbreak  [0, 2\pi / w_F)$. 
Reserve the letter $\frp$ for prime ideals. Let $v_{\frp}$ denote the $\frp$-adic valuation; occasionally, we write $ v_{p}$ for $\frp = (p)$. 

 Define $e [z] = e (\Tr (z)) $ (as usual $e (x) = \exp (2\pi i x) $) to be the standard additive character on $\BC$.   
 Let $\nd z$ be twice the   Lebesgue measure on $\BC$.  
 
 It will also be convenient to introduce $ e_F [ z ] = e [\updelta_F^{-1}  z]$.

\delete{ For  $m_1,m_2 \in\SO'$ and $c \in \SO \smallsetminus \{0\}$ define the Kloosterman sum 
 \begin{align}\label{2eq: defn Kloosterman KS}
 	S  (m_1, m_2 ; c ) = \sum_{   \valpha \shskip \in (\SO  / c \SO)^{\times}  } e  \bigg[   \frac {m_1 \valpha + m_2 \widebar{\valpha} } {c } \bigg] , 
 \end{align} 
with $  \valpha \widebar{\valpha}  \equiv 1 (\mathrm{mod} \,  c)$.
}

  For  $n_1,n_2 \in\SO$ and $c \in \SO \smallsetminus \{0\}$ define the Kloosterman sum 
\begin{align}\label{2eq: defn Kloosterman KS, 2}
	S  (n_1, n_2 ; c ) = \sum_{   \valpha \shskip \in (\SO / c \SO)^{\times} } e_F \bigg[   \frac {n_1 \valpha + n_2 \widebar{\valpha} } {c } \bigg] , 
\end{align} with  
$\valpha \widebar{\valpha}  \equiv 1 (\mathrm{mod} \, c)$.

The   (unitary) characters on $\BC^{\times} $ are of the form     \begin{align}\label{2eq: chi vk}
	\vchi_{i \varnu      ,  \vvarkappa          } (z) = |z|^{  2 i \varnu        } (z/|z|)^{  \vvarkappa          }  , 
\end{align}  for   $ \varnu         $ real and   $ \vvarkappa        $ integral, 
so we introduce  $\EuScript{A}     = \BR \times  \BZ$ to parameterize the (Mellin) unitary dual of $\BC^{\times}  $.  
Define its dual $ \widehat{\EuScript{A}} = \BR \times \BR / 2 \pi \BZ$. 
Let $\nd \mu (\varnu     ,  \vvarkappa ) $ or $\nd \widehat{\mu} (r, \omega)$  denote the usual Lebesgue measure on  $\EuScript{A}    $ or $ \widehat{\EuScript{A}}$ respectively. 

For $n, q \in \SO \smallsetminus \{0\}$, with $ 2 \nmid \RN (q) $, let 
\begin{align}\label{2eq: Jacobi}
	 \vchi_{q} (n) = \Big(\frac n q \Big)
\end{align}   
denote the  Jacobi symbol; for $(q)$ square-free, $ \vchi_{q} $ is   the quadratic Dirichlet character of conductor $(q)$. 

\subsection{Poisson Summation Formula} 

The following Poisson summation formula is a special case of \cite[VII.3 (3.2)]{Neukirch-ANT}.

\begin{lem}\label{lem: Poisson}
	Let $\varww \in C_c^{\infty} (\BC)$. Then 
	\begin{align}
		\frac 1 {\sqrt{|d_{F}|}}	\sum_{m \in \SO'} \varww (m) =    \sum_{n \in \SO } \widehat {\varww} (n),
	\end{align}
	where $ \widehat {\varww} $ is the complex Fourier transform of $ \varww $ defined to be 
	\begin{align}\label{4eq: Fourier}
		\widehat {\varww} (u) 
		= \int \hskip -5pt \int_{\BC} \varww (z) e[     {u} z] \nd z, \qquad  (\text{$u \in \BC$}).
	\end{align}
\end{lem}

\begin{cor}\label{cor: Poisson}
	Let $\varww \in C_c^{\infty} (\BC)$. For non-zero $ c \in \SO$  let $\phi (\hskip 1pt \cdot \hskip 1pt | c)$ be a function on $\SO  / c \shskip \SO$. Then
	\begin{align}\label{4eq: Poisson Cor}
	 \sum_{n \in \SO} \phi (n| c)  \varww (n) = \frac 1 {\sqrt{|d_{F}|} \RN (c)}  \sum_{ n \in \SO} \widehat {\phi }  (n|c) \widehat{\varww} (n/ \updelta_F c), 
	\end{align}
	where $\widehat{\varww} $ is as defined in {\rm\eqref{4eq: Fourier}}, and    
	\begin{align}
		\widehat {\phi }  (n|c) =  \sum_{a \in \SO  / c \shskip \SO } \phi (a|c) e_F \Big[ \hskip -1pt   - \frac {    {n} a } {c} \Big] .  
	\end{align}
\end{cor}

\begin{proof}
	We  split the left-hand side of \eqref{4eq: Poisson Cor} according to the residue classes in  $\SO  / c \shskip \SO$ as follows:
	\begin{align*}
		\sum_{a \in \SO / c \shskip \SO } \phi (a | c)     \sum_{m \in \SO'} \varww (a  + \updelta_F c m  ) . 
	\end{align*}
	By applying the Poisson summation in Lemma \ref{lem: Poisson} to the inner sum, we obtain 
	\begin{align*}
		& \sqrt{|d_{F}|} \sum_{a \in \SO / c \shskip \SO} \phi (a | c) \sum_{ n \in \SO} \viint_{\BC}  \varww ( a  + \updelta_F c z ) e [     {n} z] \nd z \\
		= & \sqrt{|d_{F}|} \sum_{a \in \SO  / c \shskip \SO} \phi (a | c) \sum_{ n \in \SO} \frac 1 {|\updelta_F c|^2 } \int \hskip -5pt \int_{\BC}  \varww (  z) e \Big[  \frac {    {n} z } {\updelta_F c}  - \frac {    {n} a } {\updelta_F c}   \Big] \nd z,
	\end{align*}
	and hence the right-hand side of \eqref{4eq: Poisson Cor}. 
\end{proof}

\subsection{The Quadratic Large Sieve of Heath-Brown} 

This paper will conclude with the quadratic large sieve inequality of Heath-Brown \cite{Heath-Brown-Quad-LS}. It has been extended to the Gaussian field and arbitrary number fields in  \cite{Heath-Brown-Quad-LS-2,Heath-Brown-Quad-LS-3}. 

\begin{lem}\label{lem:Heath-Brown LS}
	Let $  N, Q \geqslant 1$ and $\vepsilon > 0$ be real numbers.  Let  $\{ a_n \}$ be an arbitrary complex sequence on $\SO \smallsetminus \{0\}$. Then 
	$$ \sumf_{\RN(q) \leq Q} \bigg|\ \sumf_{\RN(n) \leq N} a_n \Big( \frac{n}{q} \Big)  \bigg|^2 \Lt (Q+N) (QN)^\vepsilon \sumf_{\RN(n) \leq N} |a_n|^2 ,$$
	where  the superscript {\small $ \flat$} indicates restriction to odd square-free integers in $\SO \smallsetminus \{0\}$. 
\end{lem}

\subsection{Notion of Inert Functions} \label{sec: inert functions}
For the purpose to facilitate the stationary phase analysis on $\BR$,  the notion of inert functions is introduced  in \cite{KPY-Stationary-Phase}.

Let $A [a, b ]  $ denote the annulus $  \{ z : |z| \in [a, b ]  \} $; in practice, the ratio $b/a = O (1)$. As usual let $D (a) = A[0, a]$ stand for a disc centered at $0$. 	

For $S, X > 0$,  a function $\varww   \in C^{\infty} (A[a, b])$ (not necessarily compactly supported) is said to be $S$-bounded and $X$-inert if 
\begin{align}\label{2eq: inert annulus}
	|z |^{i+j    } 	 \partial^{i+j   } \varww (z ) / \partial z^i \partial \widebar{z}^j     \Lt_{i, j    } S X^{i+j    } ,
\end{align} 
for all $i   $, $j$,  and  $z  \in A[a, b]$. Similarly, a function  $\varww   \in C_c^{\infty} (D[a ])$ is  said to be $S$-bounded and $X$-inert if 
\begin{align}\label{2eq: inert disc}
 	 \partial^{i+j   } \varww (z ) / \partial z^i \partial \widebar{z}^j     \Lt_{i, j    } S X^{i+j    } . 
\end{align}  
More generally,  a  multi-variable  function $\varww   \in C^{\infty} (A[\boldsymbol{\text{{\scalebox{0.9}[1]{$a$}}}} ,  {\boldsymbol{\text{{\scalebox{0.9}[1]{$b$}}}}}   ]  )$ is said to be $S$-bounded and $X$-inert if 
\begin{align}\label{2eq: inert annulus, 2}
	|\boldsymbol{\text{{\scalebox{0.9}[1]{$z$}}}}|^{\boldsymbol{\text{{\scalebox{0.9}[1]{$k$}}}}}  	\partial^{\boldsymbol{\boldsymbol{\text{{\scalebox{0.9}[1]{$k$}}}}   }} \varww (\boldsymbol{\text{{\scalebox{0.9}[1]{$z$}}}}) /  \partial \boldsymbol{\text{{\scalebox{0.9}[1]{$z$}}}}^{\boldsymbol{i}} \partial \widebar{{\boldsymbol{\text{{\scalebox{0.9}[1]{$z$}}}}}}^{\boldsymbol{j}} \Lt_{\boldsymbol{i   }, \boldsymbol{j}}  S X^{|\boldsymbol{\text{{\scalebox{0.9}[1]{$k$}}}}|} , \qquad \boldsymbol{\text{{\scalebox{0.9}[1]{$k$}}}} = \boldsymbol{i} + \boldsymbol{j}, 
\end{align} 
for all $\boldsymbol{i   }$, $\boldsymbol{ j  }$,  and $\boldsymbol{\text{{\scalebox{0.9}[1]{$z$}}}} \in A[\boldsymbol{\text{{\scalebox{0.9}[1]{$a$}}}} ,  {\boldsymbol{\text{{\scalebox{0.9}[1]{$b$}}}} }  ] $; 
here we have used the multi-variable notation:
\begin{align*}
	&  A[\boldsymbol{\text{{\scalebox{0.9}[1]{$a$}}}} ,  {\boldsymbol{\text{{\scalebox{0.9}[1]{$b$}}}}} ] = A[a_{1},  b_{1}] \times A[a_{2}, b_{2}] \times \cdots , \\
	&	\boldsymbol{\text{{\scalebox{0.9}[1]{$z$}}}} = (z_1, z_2, \cdots),  \quad \boldsymbol{i   } = (i_1, i_2, \cdots),  \quad \boldsymbol{j   } = (j_1, j_2, \cdots),   	 \\
	& |\boldsymbol{\text{{\scalebox{0.9}[1]{$z$}}}}|^{\boldsymbol{\text{{\scalebox{0.9}[1]{$k$}}}}} = |z_1|^{k   _1} |z_2|^{k   _2} \cdots, \quad  \frac {\partial^{\boldsymbol{\text{{\scalebox{0.9}[1]{$k$}}}}} }   { \partial \boldsymbol{\text{{\scalebox{0.9}[1]{$z$}}}}^{\boldsymbol{i}} \partial \widebar{{\boldsymbol{\text{{\scalebox{0.9}[1]{$z$}}}}}}^{\boldsymbol{j}}} = \frac {\partial^{k   _1}}  {\partial z_1^{i_1} \partial \widebar{z}_1^{j_1}}    \frac {  \partial^{k   _2} } {  \partial z_2^{i_2} \partial \widebar{z}_2^{j_2} }  \cdots, \quad |\boldsymbol{\text{{\scalebox{0.9}[1]{$k$}}}}| =  k   _1 + k   _2 + \cdots . 
\end{align*} 
Let us say simply `$X$-inert' if  $S = O(1)$ or  `inert' if moreover $X = O(1)$.

\subsection{Stationary Phase and Weil's Identity}

We record here  \cite[Lemma A.1]{AHLQ-Bessel} and  \cite[Lemma 7.4]{Qi-GL(3)}; the former is a slightly improved version of  \cite[Lemma {\rm 8.1}]{BKY-Mass}.

\begin{lem}\label{lem: staionary phase, dim 1, 2}
	Let $\varww   \in C_c^{\infty} (a, b)$. Let  $f  \in C^{\infty} [a, b]$ be real-valued.  Suppose that there
	are   parameters $P, Q, R, S, Z  > 0$ such that
	\begin{align*}
		f^{(i)} (x) \Lt_{ \, i } Z / Q^{i}, \qquad \varww^{(j)} (x) \Lt_{ \, j } S / P^{j},
	\end{align*}
	for  $i \geqslant 2$ and $j \geqslant 0$, and
	\begin{align*}
		| f' (x) | \Gt R. 
	\end{align*}
	Then 
	\begin{align*}
		\int_a^b  e (f(x)) \varww (x)  \nd x \Lt_{ A} (b - a) S \bigg( \frac {Z} {R^2Q^2} + \frac 1 {R Q} + \frac 1 {R P} \bigg)^A  
	\end{align*} 
for any  $A \geqslant 0$.
\end{lem}

\begin{lem}\label{lem: staionary phase, dim 2, 2}
	Let $D \subset \BR^2$	be a finite domain. Let $\varww \in C_c^{\infty} (D)$. Let $f \in C^{\infty} (\overline{D}) $ be real-valued. Suppose that there
	are   parameters $P, Q,  \varUpsilon, \varPhi, R, S, Z > 0$ such that
	\begin{align*}
		(\partial/\partial x)^{i} (\partial/\partial y   )^{j}	f  (x, y   ) \Lt_{   i,   j } Z / Q^{i} \varPhi ^{ j}, \quad (\partial/\partial x)^{k} (\partial/\partial y   )^{l} \varww (x, y   ) \Lt_{  k,   l }  S / P^{k} \varUpsilon ^{ l},
	\end{align*}
	for  $i, j, k, l \geqslant 0$ with $i + j \geqslant 2$, and
	\begin{align*}
	(\partial f (x, y   )/\partial x  )^2 +  (\partial f (x, y   )/\partial y    )^2   \Gt  R^2. 
	\end{align*}
	Then   
	\begin{align*}
		\begin{aligned}
			\viint_D e (f(x, y   ))  \varww (x,  y   )   \nd x \nd y  & \Lt_{    A}   \mathrm{Area}(D) S \\
			\cdot & \Bigg\{ \frac {Z^2} {R^3} \bigg( \frac 1 {Q^3} + \frac 1 {\varPhi^3} \bigg) + \frac {1} {R } \bigg(   \frac 1 {P} +   \frac 1 {\varUpsilon  } + \frac 1 {Q} + \frac 1 {  \varPhi}  \bigg) \hskip -1pt  \Bigg\}^A
		\end{aligned}
	\end{align*}
	for any  $A \geqslant 0$.
\end{lem}

\begin{cor}\label{cor: staionary phase, complex}
	Let $D \subset \BC$	be a finite domain. Let $\varww \in C_c^{\infty} (D)$. Let $f \in C^{\infty} (\overline{D}) $  be analytic. Suppose that there
	are   parameters $P, Q,  R, S, Z > 0$ such that
	\begin{align*} 
	\partial^{i}  f  ( z  ) /\partial z^{i}  \Lt_{   i  } Z / Q^{i}, \qquad \partial^{j+k} \varww ( z  )  /\partial z^{j} \partial \widebar{z}^{k} \Lt_{  j,  k}  S / P^{j + k},
\end{align*}
for  $i \geqslant 2$ with $j, k  \geqslant 0$, and
	\begin{align*}
	| \partial f (z )/\partial z |   \Gt  R . 
	\end{align*}
	Then   
	\begin{align}
		\begin{aligned}
			\viint_D e [ f( z   ) ]  \varww ( z  )    \nd z  \Lt_{    A}   \mathrm{Area}(D) S  \bigg( \frac {Z^2} {R^3 Q^3} + \frac {1} {R Q} + \frac 1 {R P}    \bigg)^A
		\end{aligned}
	\end{align}
	for any  $A \geqslant 0$. 
\end{cor}

Next we invoke the Parseval--Plancherel identity of Weil \cite{Weil-Unitary};  the formula below is a rewriting  of (26) in \cite{Jacquet-RTF}  over $\BC $:  
\begin{align}\label{2eq: Weil} 
	\viint_{\BC} \varww (z) e \bigg[   \frac {u z^2} 2   \bigg] \nd z = \frac 1 { |u|} 	\viint_{\BC} \widehat{\varww} (z) e \bigg[ - \frac { z^2} {2 u}   \bigg] \nd z,  \qquad  (\text{$u \in \BC$}), 
\end{align}
where $\varww$ is a Schwartz function and $\widehat{\varww}$ is its Fourier transform (see \eqref{4eq: Fourier}). 
It should be stressed that the Weil identity (with the Weil constant $\gamma (u)$ on the right) is also valid for $\BR$ and non-Archimedean local fields. 

For $\varww \in C_c^{\infty} (\BC)$ define 
\begin{align}\label{2eq: defn w}
	\breve{\varww} (u) = |u|	\viint_{\BC} \varww (z) e  [     {u z^2}    ] \nd z,  \qquad  (\text{$u \in \BC$}). 
\end{align}

 \begin{lem}\label{lem: Weil}
  Let $a, b, S, Y > 0$. Suppose that $ \varww \in C_c^{\infty} (D (a)) $ is $S$-bounded and $Y$-inert {\rm(}see {\rm\eqref{2eq: inert disc}}{\rm)}. 
  Then  $\breve{\varww} (u)$ is an $a^2 S Y^2$-bounded  and $(1+ Y^2/b)$-inert function on $A [ b  , \infty)$ {\rm(}see {\rm\eqref{2eq: inert annulus}}{\rm)}. 
 \end{lem}

\begin{proof}
	It follows from \eqref{2eq: Weil}  and \eqref{2eq: defn w} that
	\begin{align*}
		 \breve{\varww} (u) = \frac 1 2 \viint_{\BC} \widehat{\varww} (z) e  \bigg[ \hskip -1pt     -   \frac { z^2}   {4 u}    \bigg] \nd z. 
	\end{align*}
Note that $$ \widehat{\varww} (z) \Lt  a^2 S \lp 1+ \frac {|z|} Y \rp^{-A} , $$ by repeated partial integration. It is then easy to conclude that
\begin{align*}
	|u |^{i+j    } 	{ \partial^{i+j   } \breve{\varww} (u )} /  {\partial u^i \partial \widebar{u}^j }    \Lt_{i, j    } a^2 S Y^2 ,
\end{align*}  for any $|u| \geqslant Y^2 $; for the derivatives of $ e[-z^2/4u] = e ( - z^2/4u) e(  - \widebar{z}^2/4\widebar{u} )$ one may use (the analytic or anti-analytic variant of) the  Fa\`a di Bruno formula \cite{Faa-di-Bruno}.
\end{proof}

It is clear that if $\varww (z) = \varww (z; \boldsymbol{\text{{\scalebox{0.9}[1]{$z$}}}})$ involves other $X$-inert variables $\boldsymbol{\text{{\scalebox{0.9}[1]{$z$}}}} = (z_1, z_2, \cdots)$, then $  \breve{\varww} (u) = \breve{\varww} (u; \boldsymbol{\text{{\scalebox{0.9}[1]{$z$}}}}) $ remains $X$-inert in these variables.

The stationary phase analysis on $\BC$ usually requires extra care for the angular variable (see \cite[\S \S 2.4, 6.1]{Qi-Gauss}), but in our symmetric square case the Weil identity over $\BC$ serves well as a simpler substitute. 

\delete{

According to \cite[\S \S 1.1, 1.2]{MO-Formulas}, for   $ |\arg (s) | < \pi$, as $|s| \ra \infty$, we have 
\begin{align}\label{6eq: Stirling, 1}
	\log \Gamma (s) = \lp s+\frac 1 2 \rp \log s - s + \frac 1 2 \log (2\pi) + O \lp \frac 1 {|s|} \rp, 
\end{align}
\begin{align}\label{6eq: Stirling, 2} 
	\psi (s) = \log s + O \lp \frac 1 {|s|} \rp, \qquad \psi^{(j)} (s) = \frac {(-)^{j-1} (j-1)!} {s^{j}} + O \lp \frac 1 {|s|^{j+1}} \rp, 
\end{align}
where $\psi (s)  = \Gamma' (s) / \Gamma (s) $ is the logarithm derivative of $\Gamma (s)$. 

\begin{lem}
	 Let $ |\arg (s) | < \pi$ and $|s|$ be large. Then
	 \begin{align}\label{6eq: Stirling, 4}
	 	\frac {\nd^j} {\nd s^j }	\lp \frac {\Gamma (s+ \valpha)} {\Gamma (s + \beta)} \rp = O_{j, \mathrm{Re} (\valpha-\beta)}   \bigg(  \big|s^{\valpha - \beta}  \big|  \left(\frac {1 + |\valpha -\beta| } {|s|} \right)^j  \bigg), 
	 \end{align}
 provided  that  $|\valpha|, |\beta| \Lt \sqrt{|s|}$.
\end{lem}

\begin{proof} 
	For  $|\valpha|, |\beta| \Lt \sqrt{|s|}$, it follows from \eqref{6eq: Stirling, 1} that
	\begin{align}\label{6eq: Stirling, 3}
		\frac {\Gamma (s+ \valpha)} {\Gamma (s + \beta)} = O_{\mathrm{Re} (\valpha-\beta)}  \big( \big|s^{\valpha - \beta}  \big|  \big), 
	\end{align}
and in general   \eqref{6eq: Stirling, 4} is deducible from \eqref{6eq: Stirling, 2}  and \eqref{6eq: Stirling, 3}. 
\end{proof}}

\subsection{The Mellin Technique} 
Actually, in our later analysis, those `inert' variables which do not occur in the phase function will be inessential and suppressed from the notation. However, one needs to separate these variables prior to the application of large sieve. For this technical purpose, the following Mellin expression from \cite[Lemma 12.3]{Qi-GL(3)}\footnote{Note that $R^2$ should be removed from $O ( R^2 S T^{-A} )$ in \cite[Lemma 12.3]{Qi-GL(3)}. } will be useful.

\begin{lem}\label{lem: w(z) Mellin} 
	Let $R, S >0$ and $T, X \geq 1$. Let the function $\varww \in C^{\infty} (A [R/3, 3 R])$ be $S$-bounded and  $X$-inert. We have
	\begin{align}\label{2eq: Mellin of w(z)}
		  \varww(z) = \viint_{\EuScript{A} } \xi (2\varnu, \vvarkappa)\overline{\vchi_{i\varnu, \vvarkappa}(z) }\nd \mu(\varnu, \vvarkappa), 
	\end{align}
	for any $z \in A[R/2, 2R]$, where the function $\xi (\varnu, \vvarkappa)$ has bounds $\xi (\varnu, \vvarkappa) = O (S)  $ and $\xi (\varnu, \vvarkappa) = O ( S T^{-A} )$ if $\sqrt{\varnu^2 + \vvarkappa^2}>  X T^{\vepsilon}$. 
\end{lem}

More importantly, in order to separate the variables in the phase,  we need  the following Mellin expression of $e [z]$ from \cite[Lemma 12.4]{Qi-GL(3)}. 
\begin{lem}\label{lem: e[z] Mellin}
	Let $R \Gt 1$.    Then  
	\begin{align}\label{2eq: Mellin}
		e [z] = \viint_{ {\EuScript{A}}}  \xi_R  (2 \varnu, \vvarkappa          )  \overline{\vchi_{i \varnu,  \vvarkappa          } (z)} \nd \mu (\varnu, \vvarkappa          )  , 
	\end{align}
	for  any $ z \in A[ R/ {2} , 2 R]  $, where $ \xi_R (\varnu, \vvarkappa          ) = O ( (R + |\varnu      |+| \vvarkappa          |)^{-A} )  $ unless $ \sqrt{\varnu^2 + \vvarkappa^2} \asymp R$, in which case $  \xi_R (\varnu, \vvarkappa          ) = O ((\log R) / R) $.  
\end{lem}

Of course the $2$ and $3$ in Lemmas \ref{lem: w(z) Mellin}  and \ref{lem: e[z] Mellin}  may be replaced by larger absolute constants for our convenience. 

\begin{defn}\label{defn: square}
	For $U \Gt 1$, define the square 
	\begin{align}
		\EuScript{A}_{\vepsilon} (U) = \big\{ (\varnu, \vvarkappa) \in \EuScript{A} : |\varnu | , \, |\vvarkappa| \leqslant U T^{\vepsilon} \big\},
	\end{align}
and the annulus
\begin{align}
	\EuScript{A}  (U) = \big\{ (\varnu, \vvarkappa) \in \EuScript{A} : \sqrt{\varnu^2 + \vvarkappa^2} \asymp U \big\}. 
\end{align}
\end{defn}

The integrals in \eqref{2eq: Mellin of w(z)} and \eqref{2eq: Mellin} may be effectively restricted to $\EuScript{A}_{\vepsilon} ( X )$ and $\EuScript{A}  ( R )$ respectively.  

	\subsection{Hecke--Maass Forms on {\large$\BH^3$}}  
Finally  we briefly review the theory of Hecke--Maass forms on $  \BH^3$.  For more details, the reader is referred to \cite{EGM}, \cite[\S 2.1]{Qi-Liu-LLZ} or \cite[\S 3]{Qi-Liu-Moments}.

Let $$\BH^3 = \left\{ w = z + j r  : z = x+iy \in\BC, r \in \BR_+ \right\}  
$$ denote the 3-dimensional hyperbolic space, with the action of $\SL_2 (\BC)$  given by 
\begin{align*}
	z (g  \cdot   w) = \frac { (az+ b) (\overline {c} \overline z + \overline d)   +   a \overline c r^2 } {|cz+d|^2  +  |c|^2 r^2} , \hskip 5 pt r (g   \cdot   w) = \frac {r  } {{|cz+d|^2 + |c|^2 r^2}}, \quad g = \begin{pmatrix}
		a \hskip -1pt & b \\ 
		c \hskip -1pt & d \\
	\end{pmatrix} .
\end{align*}  
The space $\BH^3$ is equipped with the $\SL_2 (\BC)$-invariant hyperbolic metric $    ( \nd x^2 + \nd y^2 + \nd r^2) / r^2$ and hyperbolic measure $    \nd x\, \nd y\, \nd r / r^3$.  The   hyperbolic Laplace--Beltrami operator is given by $\varDelta =    r^2 \lp \partial^2/\partial x^2 + \partial^2/\partial y^2 + \partial^2/\partial r^2  \rp -   r \partial /\partial r$.

	Let $\varGamma = \mathrm{PSL}_2 (\SO)$ be the Bianchi group. For a non-zero ideal $\frq \subset \SO$, define the Hecke congruence subgroup $\varGamma_0(\frq) \subset \varGamma$ as follows:
\begin{equation*}
	\varGamma_0(\frq) = \left\{ 
	\begin{pmatrix}
		a & b \\
		c & d
	\end{pmatrix}
	\in \varGamma : \, c\equiv 0\, (\mathrm{mod}\, \frq) \right\}.
\end{equation*}
Let  $ \SB (\frq) $ be an orthonormal basis of Maass cusp forms  in the  $L^2$-cuspidal spectrum  of the  Laplace--Beltrami operator  $\varDelta$ on $  \varGamma_0(\frq) \backslash \BH^3$, with respect to the $L^2$-norm 
\begin{align*}
	\|f\|_{\varGamma_0(\frq)}^2 = \frac 1 {\mathrm{Vol} (\varGamma_0(\frq) \backslash \BH^3)} \int_{ \varGamma_0(\frq) \backslash \BH^3 } |f (w)|^2 \frac {\nd x\, \nd y\, \nd r}  {r^3} .
\end{align*} 
For $f  $ with Laplacian eigenvalue $   1 + 4  t_f^2  $, we have the Fourier expansion: 
 \begin{align*}
 	f  ( z, r ) =   \frac 1 {\sqrt{|d_F|}} \sum_{ m  \in \SO' \smallsetminus \{0\} } a_f ( \updelta_F m) r K_{ 2 i t_f} (4 \pi |m| r) e [ m z]. 
 \end{align*} 
The normalizing factors $1/\mathrm{Vol} (\varGamma_0(\frq) \backslash \BH^3)$ and $1/   {\sqrt{|d_F|}}$ are adopted from \cite{Koyama-PGT-Picard} and \cite{Venkatesh-BeyondEndoscopy,Qi-Liu-Moments}. According to the Kim--Sarnak bound in \cite{Blomer-Brumley}, we know that $t_f$ is either real or purely imaginary with $ |\mathrm{Im} (t_f) | \leqslant 7/64 $. 

Let us assume that the basis $ \SB = \SB (1) $ is contained in $  \SB (\frq) $. 

For any $\epsilon \in \SO^{\times}$ and non-zero $n \in \SO $, we have $a_f (\epsilon^2 n) = a_f (n)$ since the diagonal matrix $\mathrm{diag} (\epsilon, \epsilon^{-1})$ lies in $ \varGamma_0 (\frq) $.  It is therefore legitimate to write $ a_f (  \frn^2) =  a_f (  n^2)$ for $ \frn = (n )   $. 

Finally, we consider the Hecke--Maass forms on the Bianchi manifold $ \varGamma \backslash \BH^3 $.  For  $n \in \SO \smallsetminus \{0\} $  define the Hecke operator $T_{n}$ by
\begin{equation*}
	T_{n} f ( w  ) = \frac 1 {w_F \sqrt {\RN (n)}} \sum_{ a d = n} \ \sum_{ b (\mathrm{mod}\, d)} f \lp \begin{pmatrix}
		a & b \\
		& d
	\end{pmatrix}
	w \rp .
\end{equation*}
Hecke operators commute with each other as well as  the Laplace--Beltrami operator. Thus we may further assume that every $f \in \SB  $ is an eigenfunction of all the Hecke operators $T_n$. Let $\lambda_f (n)$ denote the Hecke eigenvalue of $T_n$. It is known that $ \lambda_f (n)$ are all real.  As usual $ \lambda_f (n) $ and $a_f (  n)$ are proportional:
\begin{align}\label{2eq: rho(n) = lambda(n)}
a_f   (  n) = a_f   ( 1 )\cdot	\lambda_f (n) . 
\end{align}
By the discussion above, one may write $ \lambda_f (\frn^2) = \lambda_f (n^2) $ for $ \frn = (n) $.

\section{Initial Reductions} 

First, by standard  argument using the approximate functional equation and the Kuznetsov trace formula, we would like to reduce the problem to the estimation of certain Kloosterman--Bessel sums as in Proposition \ref{prop: initial reduction}. 

Similar to \cite{Blomer-Hecke-Quad,Khan-Young-Sym2}, for technical reasons we shall embed $ \SB \hookrightarrow \SB (16)$ and use the Kuznetsov formula for $ \varGamma_0 (16) $. See Remark \ref{rem: 4 | n1n2}.  In order to lighten the burden of notation,  the contributions from  Eisenstein series  will  not be written down explicitly (as will be dropped entirely by non-negativity). The reader may find  the Kuznetsov formula for  $\varGamma_0 (\frq)$ in \cite[Theorem 6.1]{BM-Kuz-Spherical}, \cite[Proposition 1]{Venkatesh-BeyondEndoscopy}, or \cite[Theorem 11.3.3]{B-Mo2},  though the spectral sides therein are not as explicit as that in  \cite[Theorem 16.3]{IK}\footnote{\cite[Proposition 2.1]{Qi-Liu-LLZ} is more explicit but the level $\frq$ is assumed to be square-free.}.

\subsection{Symmetric Square $L$-Functions} \label{sec: L-function}

\delete{and express the Rankin--Selberg $L (s, f \times f)$ as 
	\begin{align*}
		L (s, f \times f) = \zeta_F (s) \zeta_F (2s)^{-1} L (s, \mathrm{Sym}^2 f ) . 
\end{align*}}

Let $f \in \SB$.  For $\mathrm{Re} (s) > 1$, the symmetric square $L$-function is defined by
\begin{align*}
	L (s, \mathrm{Sym}^2 f ) = \sum_{\frn \subset \SO} \frac {c_{  f}  (\frn) } {\RN(\frn)^{s} }  ,  \qquad c_{  f}  (\frn) = \sum_{ \mathfrak{l}^2 \mathfrak{m} = \frn } \lambda_{\uppii} (\frm^2).  
\end{align*}
It will be more convenient to use the expression: 
\begin{align*}
	L (s, \mathrm{Sym}^2 f ) = \zeta_{F} (2s) \sum_{\frn \subset \SO} \frac {\lambda_{f} (\frn^2) } {\RN(\frn)^{s}}, 
\end{align*}
where $\zeta_F (s)$ is  the Dedekind $\zeta$ function  $$ \zeta_F (s) = \sum_{\frn \subset \SO} \frac {1 } {\RN(\frn)^{s}} . $$  
The gamma factor of $\mathrm{Sym}^2 f $ is equal to $ (2\pi)^{-3 s} \gamma (s, t_f ) $, with 
\begin{align*}
	\gamma (s, t ) = \Gamma (s) \Gamma (s+2it  ) \Gamma (s-2it   ). 
\end{align*}
It is known   that $ L (s, \mathrm{Sym}^2 f ) $ is entire (\cite{Shimura-Sym2,GJ-GL(2)-GL(3)}) and satisfies the functional equation
\begin{align*}
	\Lambda (s, \mathrm{Sym}^2 f ) = \Lambda (1-s, \mathrm{Sym}^2 f ), 
\end{align*}
where $$ \Lambda (s, \mathrm{Sym}^2 f ) = \sigma_F^{-3s} \gamma (s, t _f ) L (s, \mathrm{Sym}^2 f) , \qquad \sigma_F = 2\pi / \sqrt{|d_F|}.  $$

\subsection{Approximate Functional Equation} 
Let $\mathrm{Re} (s) = 1/2$ and $|s| < T$. 
Suppose that  $||t_f|-T | <  T^{1-\vepsilon}$ ($t_f$ real). 

By the approximate functional equation in \cite[Theorem 5.3]{IK} (see also \cite[Lemma 3.1]{Iwaniec-Michel-Sym2}), we have
\begin{align*}
	L (s, \mathrm{Sym}^2 f) =	\sum_{ \frn \subset \SO }  \frac {   \lambda_{f}  (\frn^2 )   } { \RN (\frn)^s  }     V_{s}     ( \sigma_F^3  \RN(\frn) ; t_f  ) + \epsilon_F (s, t_f) \sum_{ \frn \subset \SO }  \frac {   \lambda_{f}  (\frn^2 )   } { \RN (\frn)^{1-s}  }     V_{1-s}     (  \sigma_F^3  \RN(\frn)  ; t_f )  , 
\end{align*}
where
\begin{align*}
	\epsilon_F (s, t) = \sigma_F^{3(2s-1)} \frac {\gamma (1-s, t)} {\gamma (s, t)}   
\end{align*}
has  norm unity (for  $\mathrm{Re} (s) = 1/2$). 
\cite[Proposition 5.4]{IK} implies that  one may effectively restrict the sums above to the range  $ \RN (\frn) \leqslant  |s| T^{2+\vepsilon} $ at the cost of a negligibly small error. In order to separate  $t_f$ from $ \RN (\frn) $, we use  the following expression  due to Blomer \cite[Lemma 1]{Blomer}  (slightly modified): 
\begin{align*}
	V_s (y; t) =  \frac 1 {2   \pi i }   \int_{   \vepsilon - i U}^{\vepsilon + i U}    y^{ - v}  \zeta_{F} (2s+2v) \frac {\gamma (s+ v, t      )} {\gamma (s, t      )} \exp ({v^2} )    \frac {\nd v} {v} +  {O_{\vepsilon}} \bigg( \frac { T^{  \vepsilon} } {y^{ \vepsilon} \exp ({U^2/2}) } \bigg) .
\end{align*}
The error term above is negligibly small if we choose $ U = \log T$. Note that  for any $v$ on the integral contour  
\begin{align*}
	\frac {\gamma (s+ v, t    )} {\gamma (s, t    )}  \exp ({v^2} )  = O_{\vepsilon}  ( T^{\vepsilon} )  , 
\end{align*} 
by the Stirling formula.  On a smooth dyadic partition and Cauchy--Schwarz, we infer that up to a negligible error 
\begin{align}\label{2eq: AFE} 
	| L (s, \mathrm{Sym}^2 f) |^2 \Lt  {T^{\vepsilon}}     \sum_{N \leqslant |s|T^{2+\vepsilon}}  \frac{ 1 } {N } \int_{    \vepsilon - i  \log T}^{\vepsilon + i \log T} \bigg| \sum_{ \RN (n) \sim N }    {  \lambda_f (  n^2)   } \varww_{s, v} \bigg( \frac {n} { \sqrt{N}} \bigg)        \bigg|^2 \nd v ,
\end{align}
where   $N$ are dyadic in the form $ 2^{k/2}  $ ($k \geqslant -1$), and
\begin{align*}
	\varww_{s, v} (z) =   {\varvv (|z| )} / {|z|^{2s+2v} } 
\end{align*}
for a certain fixed $\varvv \in C_c^{\infty} [1,  \sqrt{2}]$. Note that  $  \varww_{s, v}  (z)$ is $X$-inert (see \S \ref{sec: inert functions}) if we set
\begin{align}\label{3eq: X}
	X = |s| + \log T . 
\end{align} 

\subsection{Application of Kuznetsov} 

For $ T^{\vepsilon} \leqslant M \leqslant T^{1-\vepsilon} $ define the spectral weight function 
\begin{align*}
	h (t) =   \exp \hskip -1pt \lp - \frac {(t-T)^2} {M^2}   \rp + \exp \hskip -1pt \lp - \frac {(t+T)^2} {M^2}   \rp,
\end{align*}
and  consider
\begin{align}\label{3eq: sum, 1}
	\sum_{f \in \SB}  \omega_f h (t_f) |L(s, \mathrm{Sym}^2 f) |^2, 
\end{align}
where the harmonic weight $\omega_f $ is given by 
\begin{align*}
	\omega_f = \frac { t_f |a_f   ( 1 )  |^2 } {\sinh (2\pi t_f)} ,
\end{align*}
which is harmless as $ 1/ \omega_f $ is a multiple of $ L (1, \mathrm{Sym}^2 f)$ by Rankin--Selberg,  while $$L (1, \mathrm{Sym}^2 f) \Lt (1+ |t_f|)^{\vepsilon}$$  
by \cite[Theorem 1]{Molteni-L(1)}. Our aim is to prove that \eqref{3eq: sum, 1} is bounded by $ O_{\vepsilon} ( { |s|^{12} T^{4+\vepsilon}} /{M}) $ for $|s| < T$.

We start by inserting into \eqref{3eq: sum, 1} 
the bound for $|L(s, \mathrm{Sym}^2 f)|^2$ as in \eqref{2eq: AFE},  with $\lambda_f (\frn^2)$ replaced by $ a_f (  \frn^2) / a_f   ( 1 )  $ as in \eqref{2eq: rho(n) = lambda(n)} (implicitly, our attention has been restricted to the range $ | |t|  - T| < M^{1+\vepsilon}$ since $h (t )$ is exponentially small if otherwise). However,   the integral on $v$ and the sum on dyadic $N$ will not be in our concern as  by trivial estimation we only lose $O  (\log^2 T)$.    Indeed, the problem is reduced to prove 
\begin{align}\label{3eq: bound, 1}
 	\frac 1 {N} \sum_{f \in \SB} \frac { t_f  h (t_f) } {\sinh (\pi t_f)}   \bigg|\sum_{ \RN (n) \sim N }    {  a_f (  n^2)   } \varww \bigg(  \frac {n} { \sqrt{N}} \bigg)  \bigg|^2 \Lt \frac{ X^{12} T^{4+\vepsilon}} {M} ,  
\end{align} 
(uniformly) for any   $
	 N <  {X} T^{ 2+\vepsilon}$,  $X < T$   
(see \eqref{1eq: s<T} and \eqref{3eq: X}),   and any $X$-inert function $\varww (z)$ supported on $ A [1,  \sqrt{2}]$ (see \S \ref{sec: inert functions}). 

Next we extend the sum in \eqref{3eq: bound, 1} onto $\SB (16) $ and include the integrals from the Eisenstein series.   
By opening the square, we obtain 
\begin{align}\label{3eq: spectral}
\frac 1 {N} \mathop{\sum \sum}_{ n_1, n_2 } \varww \bigg( \frac {n_1} { \sqrt{N}} \bigg)  \overline{\varww} \bigg(  \frac {n_2} { \sqrt{N}} \bigg) 	\Bigg\{ \sum_{f \in \SB (16)} \frac { t_f h (t_f) } {\sinh (2\pi t_f)} a_f     ( n_1^2) \overline{a_f ( n_2^2)} + \mathrm{CSC} \Bigg\}   , 
\end{align}
where CSC stands for the continuous spectral contribution. 
The Kuznetsov formula says that up to a constant multiple the expression between the large parentheses  in  \eqref{3eq: spectral} equals 
\begin{align}\label{3eq: Kuznetsov}
	  {2\sqrt{|d_F|}}  \,  \delta_{ (n_1 ),  (n_2)}   \,  \SDH +   \sum_{16 \shskip | \shskip c }  \frac {S (n_1^2, n_2^2; c)} {\RN (c)} \SDH \bigg( \frac {n_1 n_2} { \updelta_F c} \bigg), 
\end{align}
where $ \delta_{ (n_1 ),  (n_2)}  $ is the  Kronecker $\delta$ symbol (for ideals),  $ \SDH  $ and $ \SDH (z)$ are the Plancherel and Bessel integrals defined by  
\begin{align}\label{1eq: defn Bessel integral}
	\SDH = \int_{-\infty}^{\infty} h (t) t^2 \nd \shskip t , \qquad \SDH (z) = \int_{-\infty}^{\infty} h (t)  	\boldsymbol{J}_{i t} ( z ) t^2 \nd \shskip t,
\end{align} 
with the Bessel kernel
\begin{equation}\label{2eq: Bessel Bs(z)}
	\boldsymbol{J}_{\vnu} (z ) =  \frac {2\pi^2} {\sin ( 2\pi \vnu) } \big( { \textstyle  J_{- 2 \vnu} (4 \pi   {z}) J_{- 2 \vnu} (4 \pi   { \widebar{z}}) - J_{  2\vnu} (4 \pi   {z}) J_{ 2\vnu} (4 \pi   { \widebar{z}}) } \big). 
\end{equation}  
Consequently, the diagonal term contributes $ O (   M T^{2+\vepsilon}) $, while the off-diagonal sum reads 
\begin{align}\label{3eq: off-diag, 1}
\frac 1 {N}  \sum_{ 16 \shskip | \shskip c } \mathop{\sum \sum}_{n_1, n_2}  \frac {S (n_1^2, n_2^2; c)} {\RN (c)} \varww \bigg( \frac {n_1} { \sqrt{N}} \bigg)  \overline{\varww} \bigg(  \frac {n_2} { \sqrt{N}} \bigg)  \SDH \bigg( \frac {n_1 n_2} { \updelta_F c} \bigg) .  
\end{align}

\subsection{Further Reductions}
The Bessel integral has been studied in \cite[\S 3]{Qi-Liu-LLZ}. In particular, by the remark below   \cite[Lemma 3.2]{Qi-Liu-LLZ}  we have  $\SDH (z) = O (M T^{-A} |z|^2 )$ for $|z| \leqslant 1$ and by  \cite[Lemma 3.5 (1)]{Qi-Liu-LLZ} we have $\SDH (z) = O (T^{-A})$ for $1 < |z| \Lt T$. Thus we may further restrict the $c$-sum to the range $ |c| \Lt  {N /T} $ at the cost of a negligible error. It will also be  convenient to apply a dyadic partition to the $c$-sum. 

More importantly, we have the integral representation from  \cite[Lemma 3.4]{Qi-Liu-LLZ}:
\begin{align}\label{3eq: H(z)} 
	\SDH (z) =    MT^2  \viint_{ \widehat{\EuScript{A}} }  g (M r) e (  2 Tr/\pi +  4 \mathrm{Re} (z  \trh (r, \omega)  ))  \nd \widehat{\mu} (r, \omega)  , 
\end{align}
where $g (r)$ is a Schwartz function and  $ \trh (r, \omega)$ is the trigonometric-hyperbolic function 
\begin{align}\label{3eq: trh}
	 \trh (r, \omega) = \cosh r \cos \omega + i \sinh r \sin \omega. 
\end{align}
For simplicity, we prefer not to truncate the $r$-integral at $|r| = M^{\vepsilon} / M$. 
Finally, we  extract the exponential factor $e_F [ 2 n_1 n_2/ c]$ from the Bessel integral $\SDH ({n_1 n_2} / { \updelta_F c})$ to join the Kloosterman sum $S (n_1^2, n_2^2; c)$. To this end, let us introduce the variant Bessel integral:
\begin{align}\label{3eq: defn of I(z)}
	\SDI (z) = MT^2   \viint_{ \widehat{\EuScript{A}} }  g (M r) e (  2 Tr/\pi +  2 \mathrm{Re} (z  \psi (r, \omega)   ))  \nd \widehat{\mu} (r, \omega) ,
\end{align}
with 
\begin{align}\label{3eq: defn of psi}
	\psi (r, \omega) = 2 (\trh (r, \omega) - 1). 
\end{align}

By the discussions above, it is now reduced to prove the following proposition.

\begin{prop}\label{prop: initial reduction}
	 Let  $1 \Lt X < T$,   $ 1 \Lt N < X T^{ 2+\vepsilon}$, and $ 1 \Lt C \Lt (N/T)^2 $. Let   $\varww (z)$ be an $X$-inert function supported on $ A [1,  {2}]$ {\rm(}see {\rm\S \ref{sec: inert functions}}{\rm)}.  Define 
	 \begin{align}\label{3eq: S(N)}
	 	S (N, C) = \frac 1 {N} \sum_{ 16 \shskip | \shskip c \, : \, \RN (c) \sim C } \mathop{\sum \sum}_{\RN(n_1), \RN (n_2) \sim N}  \frac {S (n_1^2, n_2^2; c) e_F [2n_1n_2/c]} {\RN (c)}   \varww  \bigg( \frac {n_1} { \sqrt{N } },  \frac {n_2} { \sqrt{N } }; \frac {N } { \updelta_F c}  \bigg),
	 \end{align}
 for  
 \begin{align}\label{3eq: defn of w}
 	\varww (z_1, z_2; \varLambda) = \varww (z_1) \overline{\varww} (z_2) \SDI (\varLambda z_1 z_2). 
 \end{align}
Then 
\begin{align}
	S (N, C) \Lt  \frac{ X^{12} T^{4+\vepsilon}} {M}.  
\end{align}
\end{prop}

\begin{remark}\label{rem: c-length}
In   the setting of $\BH^2$ the range of $C$ reads: $C \Lt N^2 M^{\vepsilon} / T M$ {\rm(}see {\rm(6.15)} in {\rm\cite{Khan-Young-Sym2}}{\rm)}, where $N$ is up to $\sqrt{X} T^{1+\vepsilon}$,  so their subsequent applications of Poisson are more effective due to  the factor $M^{\vepsilon}/ M$. The absence of $M$ in the length  of summations in our setting of $\BH^3$ makes it hopeless to breach the final bound $ O_{s} (T^{3+\vepsilon}) $ as in Proposition \ref{prop: main}.  
\end{remark}

\section{Application of Double Poisson} 

By applying the Poisson summation formula in Corollary \ref{cor: Poisson} to both the $n_1$- and $n_2$-sums in \eqref{3eq: S(N)}, we infer that
\begin{align}\label{4eq: S(N, C), Poisson}
	S (N, C) = \frac {N  } {|d_F|} \sum_{ 16 \shskip | \shskip c \, : \, \RN (c) \sim C }  \frac 1 {\RN (c)^3}  \mathop{\sum \sum}_{n_1, \, n_2}  T (n_1, n_2; c)  \widehat{\varww}  \bigg( \frac{n_1 \sqrt{N}}{\updelta_F c}, \frac{n_2 \sqrt{N}}{\updelta_F c} ; \frac {N } { \updelta_F c}  \bigg) , 
\end{align}
with
\begin{align}\label{4eq: Fourier T}
	T (n_1, n_2; c) = \mathop{\sum \sum}_{\valpha_1, \valpha_2 \shskip \in \SO/c \shskip \SO} S (\valpha_1^2, \valpha_2^2; c) e_F \bigg[  \frac{ 2 \valpha_1 \valpha_2 - n_1 \valpha_1  - n_2 \valpha_2  }{c} \bigg],  
\end{align}
\begin{align}\label{4eq: Fourier w} 
	\widehat{\varww} (u_1, u_2 ; \varLambda ) = \viint \viint	\varww (z_1, z_2; \varLambda) e[ u_1 z_1 + u_2 z_2 ] \nd z_1 \nd z_2 . 
\end{align}

\section{Formula for the Exponential Sum}


For the exponential sum $T (n_1, n_2; c)$ defined as in \eqref{4eq: Fourier T},  we start by opening the Kloosterman sum (see \eqref{2eq: defn Kloosterman KS, 2}) and completing the square. It follows that 
\begin{align*}
	T (n_1, n_2; c) = \mathop{\sum \sum}_{\valpha_1, \valpha_2 \in \SO/c \shskip \SO} \sum_{\beta \shskip  \in (\SO /c \shskip\SO)^\times } e_F \bigg[ \frac{    (\valpha_1 +  \valpha_2 \widebar{\beta} )^2 \beta -n_1 \valpha_1  - n_2 \valpha_2 }{c} \bigg]. 
\end{align*}
By the substitution $\valpha_1 = \valpha - \valpha_2 \widebar{\beta}$ and an evaluation of the $\valpha_2$-sum, this is further transformed into
\begin{align}\label{5eq: T, 1}
	T (n_1, n_2; c) = \RN (c) \mathop{\sum}_{\valpha \shskip \in \SO/c \shskip \SO} \mathop{\sum_{\beta \shskip  \in (\SO /c \shskip\SO)^\times }}_{ n_2 \beta \equiv n_1 (\mathrm{mod}\, c)  } e_F \bigg[ \frac{  \valpha^2 \beta  - n_1 \valpha  }{c} \bigg]. 
\end{align}
The congruence in the inner sum implies that $ T( n_1, n_2; c) = 0 $ unless 
\begin{align}\label{5eq: (n,c)}
	 (n_1 , c) = (n_2,c) . 
\end{align}
By the Chinese remainder theorem, $T (n_1, n_2; c)$ satisfies the twisted-multiplicative relation:
\begin{equation}\label{5eq: twisted mult}
	T( n_1, n_2; c_1 c_2) = T(  \widebar{c}_2 n_1 , n_2; c_1) T(\widebar{c}_1  n_1 , n_2; c_2),   
\end{equation}
for $(c_1,c_2) = 1$. 
In particular, 
\begin{equation}\label{5eq: twist by units} 
	T( n_1, n_2; \epsilon c ) = T( \bar{\epsilon} n_1   , n_2; c ) , 
\end{equation}
for any unit $\epsilon$ in $\SO^{\times}$.

Define 
\begin{equation}\label{5eq: p2}
	\begin{split}
		  p_2 = & \left\{ \begin{aligned} 
		& 2, & & \text{ if } d_F = -3,  -11, -19, -43, -67, -163, \\
		& \displaystyle 1+i, \, \sqrt{2} i ,  & & \text{ if } d_F = -4, \, -8,  
	\end{aligned} \right. \\
	&  \ \  p _2 = \frac {1 + \sqrt{7} i} 2, \quad q_2 = \frac { 1 - \sqrt{7} i} 2, \quad \text{ if } d_F = -7, 
	\end{split}
\end{equation}
in the cases when $2 $ is inert, ramified, and  split  (as $h_F = 1$). For notational simplicity, we shall proceed in the non-split case when $d_F  \neq -7$ so that we may write (uniquely) 
\begin{align} \label{5eq: c = p2 co}
	c = p_2^{k} c_{\mathrm{o}}, \qquad 2 \nmid \RN (c_{\mathrm{o}})  ; 
\end{align} for the case $d_F = -7$,  this just needs to be modified into $c = p_2^{k} \shskip q_2^{r} c_{\mathrm{o}}$.  Next, we split $c_{\mathrm{o}}$ into 
\begin{align}\label{5eq: co = cs c*}
	c_{\mathrm{o}} = c_{\square }^2 c^{\star} ,
\end{align}
where $c^{\star}$ is the square-free part of $c_{\mathrm{o}}$.  
Note that this expression is unique if we assume $\arg (c_{\square}) \in [0, 2\pi/w_F)$. 

In Appendix \ref{app: local exp sum},  we shall prove the following analogue of \cite[Lemma 5.3]{Khan-Young-Sym2}, with resort to the  explicit formulae for Gauss sums on quadratic fields in \cite{BS-Gauss-Sums}.  

\begin{lem}\label{lem: exp sum} 
 	Let $n_1, n_2 \in \SO$ and $c \in \SO \smallsetminus \{0\}$. 
Write $c = p_2^{k} c_{\mathrm{o}} $ and $c_{\mathrm{o}} = c_{\square }^2 c^{\star}$ {\rm(}uniquely{\rm)} as in {\rm\eqref{5eq: c = p2 co}} and {\rm\eqref{5eq: co = cs c*}}. Assume $16 | c$.   We have $ T (n_1, n_2; c) = 0 $  unless
	\begin{align}\label{5eq: non-vanishing}
		 (n_1, c) = (n_2, c) \text{{\rm(}$= (d)${\rm)}}, \qquad 4 | (n_1, n_2), \qquad 
	\end{align} 
	in which case, if we 
	set  $		m_1 = n_1 / d$,  $ m_2 = n_2 / d $, $ d = p_2^{l} d_{\mathrm{o}} $, and \begin{align}
	\delta = k \, (\mathrm{mod} \, 2),  \qquad 	  \gamma  = k -  2 \lceil l /2 \rceil  \,  (\mathrm{mod} \, 4)  ,   
	\end{align}
then 
\begin{align}\label{5eq: explicit T(n; c)}
	T (n_1, n_2; c) =  \RN (c)^{3/2}  \RN (d)   \cdot     
	h_{\delta} (c_{\mathrm{o}}, d_{\mathrm{o} } ) g_{\gamma  } (m_1m_2 \widebar{c}_{\mathrm{o} } )   \cdot   \Big( \frac {m_1 m_2} {c^{\star}} \Big)    e_F   \hskip -1pt    \left[ - \frac {n_1 n_2} {4 c} \right]  , 
\end{align}
where  
$ h_{\delta} (c_{\mathrm{o}}, d_{\mathrm{o} } )$  and  $ g_{\gamma  } (n) $ are bounded, with $ h_{\delta} (c_{\mathrm{o}}, d_{\mathrm{o} } ) = 0 $ if $ d_{\mathrm{o} }  \hskip -1.5pt      \nmid \hskip -1.5pt  c_{\square }^2 $\shskip{\rm;} more explicitly, 
\begin{align}\label{5eq: h(c,d)}
		h_{\delta} (c_{\mathrm{o}}, d_{\mathrm{o}}) =  \epsilon (c^{\star})    \Big( \frac {p_2} {  c^{\star} } \hskip -1pt  \Big)^{\delta}    \mathop{\prod_{ \frp | (c_{\square}) , \frp    \nmid    (c^{\star}) } }_{ v_{\frp} (c_{\mathrm{o}}) = v_{\frp} (d_{\mathrm{o}}) } \bigg(1 - \frac 1 {\RN (\frp) } \bigg)  ,
	\end{align}
if $ d_{\mathrm{o} } |  c_{\square }^2 $ {\rm(}see {\rm\eqref{app: epsilon(c)}} for the definition of $\epsilon (c^{\star})${\rm)},  and  $ g_{\gamma  } (n)$ is a function on $\SO / 8 \shskip \shskip \SO$ whose form may alter  in the special cases $ k -  2 \lceil l /2 \rceil = 0, 1, 2, 3 $  {\rm(}see {\rm\eqref{app: defn g (n)}}, {\rm\eqref{app: g(n/p)}}--{\rm\eqref{app: last g(n/pk)}}, and {\rm\eqref{app: defn g0(n)}}{\rm)}. In particular, we have
\begin{align}\label{5eq: T(0,n;c)}
	T (0, n; c) =  \left\{ \begin{aligned}
	&	\RN (c )^{3/2} \varphi (c), & & \text{ if } (c) \text{ is square, and } c \shskip | n,    \\
	& 0 ,  & & \text{ if otherwise,}
	\end{aligned} \right.
\end{align}
where as usual $ \varphi (c) $ is the Euler toitent function. 
\end{lem}

Moreover, in the case $ l < k $ so that $m_1$ and $m_2$ are odd,  we need to separate $m_1$ and $m_2$ in $g_{\gamma} (m_1 m_2 \widebar{c}_{\mathrm{o}})$ via the Mellin decomposition: 
\begin{align}\label{5eq: g (n)}
	g_{\gamma} (n) = \sum_{\omega \shskip (\mathrm{mod}\, 8\shskip \SO)}  \widehat{g}_{\gamma} (\omega) \omega (n), 
\end{align}
where the sum is over the   characters of $ {(\SO / 8 \shskip \SO){}^{\times}}$. The case $l =k$ is simpler as $g_0 (n)$ is constant   $1- 1/ \RN (p_2)$ or  $0$ according as $\delta = 0 $ or $1$ (see \eqref{app: defn g0(n)}). 

\begin{remark} 
Although it does not affect much our analysis, slight simplification can be made  in the case $p_2 = 2$ as now $ g _{\gamma} (n) $ only depends on $ \delta = k \, (\mathrm{mod} \, 2)$ {\rm(}see {\rm\eqref{app: g(n), inert}}{\rm)} and  it changes (to constant $0$, $3/4$, or $1$) only if  $ k -  2 \lceil l /2 \rceil = 0, 1  $.  Also $ g _{\gamma} (n) $ is multiplicative in every case of $k$ and $l$, so {\rm\eqref{5eq: g (n)}} is not required.  
\end{remark}

Finally, as the exponential factor in \eqref{5eq: explicit T(n; c)} will be extracted and join the Fourier integral (see \eqref{6eq: defn wn}), let us define 
\begin{align}\label{5eq: Tn(n; c)}
T^{\natural} (n_1, n_2; c) =	T (n_1, n_2; c)     e_F   \hskip -1pt    \left[   \frac {n_1 n_2} {4 c} \right]  . 
\end{align}


\section{Analysis for the Fourier Integral} 

By \eqref{3eq: defn of I(z)}, \eqref{3eq: defn of psi}, \eqref{3eq: defn of w}, and \eqref{4eq: Fourier w}, we rewrite the Fourier integral $\widehat{\varww} (u_1, u_2 ; \varLambda )$ as follows:
\begin{align}\label{6eq: integral w}
	 \widehat{\varww} (u_1, u_2 ; \varLambda ) = M T^2 \viint_{ \widehat{\EuScript{A}} } g (Mr) e (2Tr/\pi)  \varvv ( u_1, u_2; \varLambda \psi (r, \omega)   ) \nd \widehat{\mu} (r, \omega), 
\end{align} 
with
\begin{align}
\varvv ( u_1, u_2; \zeta ) =	\viint \viint	\varww (z_1) \overline{\varww} (z_2) e[ \zeta z_1 z_2   +  u_1 z_1 + u_2 z_2  ] \nd z_1 \nd z_2 . 
\end{align}
First, note the trivial bound 
\begin{align}\label{6eq: trivial bound w}
	 \widehat{\varww} (u_1, u_2 ; \varLambda ) \Lt T^2. 
\end{align} For the zero frequency, we shall need   stronger bounds for $\widehat{\varww} (0, 0 ; \varLambda )$. 
\begin{lem}\label{lem: w (0, 0)}
Let $|\varLambda | \Gt T$. 	Then we have bounds
	\begin{align}\label{6eq: bound w(0, 0)}
		 \widehat{\varww} (0, 0 ; \varLambda ) \Lt \frac { X M T^{2+\vepsilon} } { {|\varLambda| }},
	\end{align}
and $ \widehat{\varww} (0, 0 ; \varLambda ) = O (T^{-A})$ in the case $ |\varLambda | \leqslant T^{2-\vepsilon} / X $.  
\end{lem}

Moreover, we modify $\widehat{\varww} (u_1, u_2 ; \varLambda )$ by the exponential factor arising from the formula of $ T (n_1, n_2; c) $ as in \eqref{5eq: explicit T(n; c)} (see also \eqref{5eq: Tn(n; c)}):
\begin{align}
	 \label{6eq: defn wn} 
	  \widehat{\varww}{}^{\natural}  (u_1, u_2 ; \varLambda ) = e \left[ - \frac {u_1 u_2} {4 \varLambda}  \right] \widehat{\varww} (u_1, u_2 ; \varLambda ) .  
\end{align} 
In \S  \ref{sec: Fourier} and \S \ref{sec: Mellin},   we shall analyze  $   \varvv ( u_1, u_2; \zeta ) $ and $ \widehat{\varww}{}^{\natural} (u_1, u_2 ; \varLambda )$ respectively by twice the Weil identity (Lemma \ref{lem: Weil}) and by the Mellin technique (Lemmas \ref{lem: w(z) Mellin} and \ref{lem: e[z] Mellin}).

\begin{lem}\label{lem: v (u; theta)}
	 Suppose that  $|u_1 | + |u_2| > X  T^{\vepsilon}$. Then
	 \begin{align}\label{6eq: integral v}
	 	 \varvv ( u_1, u_2;  \zeta  ) = \frac{1 }{ | u_1 u_2 |  } e    \hskip -1pt   \left[ - \frac{u_1 u_2}{\zeta}  \right] \breve{\varvv} (   \zeta)  + O (T^{-A}),
	 \end{align}
 where 	  $\breve{\varvv} ( \zeta) = \breve{\varvv} (   \zeta;   u_1, u_2)$ is $X^4$-bounded and $X$-inert  for $\zeta$ on $A[XT^{\vepsilon}, \infty)$ and also for  $u_1 , u_2 $ both on  $ A [ \sqrt[4]{2} |\zeta| / 2, 2\sqrt[4]{2} |\zeta| ]$,  and  negligibly small if $\zeta$ or $u_1$, $u_2$ is not in the indicated range.   
\end{lem}

\begin{defn}\label{defn: rectangle}
	For $V, U \Gt 1$, define rectangles
	\begin{align}
	 	\EuScript{A}  (V, U) & = \big\{ (\varnu, \vvarkappa)   : |\varnu | \leqslant V   , |\vvarkappa| \leqslant U  \big\}, \\
	 \label{6eq: defn A(V, U)}	\EuScript{A}_{\vepsilon} (V, U) & = \big\{ (\varnu, \vvarkappa)   : |\varnu | \leqslant V T^{\vepsilon}, |\vvarkappa| \leqslant U T^{\vepsilon}  \big\}. 
	\end{align} 
\end{defn}

\begin{lem}\label{lem: w (u; Lambda)}
	 Let   $|u_1 | + |u_2| > X  T^{\vepsilon}$. Let $  \rho  \Lt 1$. 
	 Then for $   |T/\varLambda|^2 \sim  \rho   $ we may write 
	 \begin{align}\label{6eq: w = w rho} 
	 	\widehat{\varww}{}^{\natural}  (u_1, u_2 ; \varLambda ) = \widehat{\varww}{}^{+}_{0 }  (u_1, u_2 ; \varLambda )   + \widehat{\varww}{}^{+}_{\rho}  (u_1, u_2 ; \varLambda ) + \widehat{\varww}{}^{-}_{\rho}  (u_1, u_2 ; \varLambda ) +  O (T^{-A}), 
	 \end{align}  
 if $\sqrt{\rho} \Gt {X T^{\vepsilon}}/ {T }$, or simply 
 \begin{align}\label{6eq: w = w rho, 1} 
 	\widehat{\varww}{}^{\natural}  (u_1, u_2 ; \varLambda ) = \widehat{\varww}{}^{+}_{0 }  (u_1, u_2 ; \varLambda )  +  O (T^{-A}), 
 \end{align}  
 if  $\sqrt{\rho} \Lt {X T^{\vepsilon}}/ {T }$, 
so that $ \widehat{\varww}{}^{\pm}_{\rho }  (u_1, u_2 ; \varLambda ) $ {\rm(}$\rho =0$ is admissible in the $+$ case{\rm)} 
   is supported on  
  \begin{align}\label{6eq: support in u}
 |u_1| , |u_2| \asymp \left\{ \begin{aligned} & |\varLambda | , & & \text{ if } +, \\
 	&   |\rho \varLambda |  , & & \text{ if } -,
 \end{aligned} \right.
  \end{align} 
 of the form 
	 \begin{align}\label{6eq: w = Phi}
	 	\widehat{\varww}{}^{\pm}_{\rho}   (u_1, u_2 ; \varLambda ) = \frac{M T^2} {|u_1 u_2 |  }  \Upsilon_{\rho}^{\pm}  \Big( \frac {u_1 u_2} {4 \varLambda} \Big) , 
	 \end{align} 
 with 
 \begin{align}\label{6eq: Phi o}
 	\Upsilon_{0}^{+} (z) =       \frac {X^5 \sqrt{\rho}\, T^{\vepsilon}}  {T}  \viint_{ \EuScript{A}_{\vepsilon} (X) } \vlambda_{0}^{+} (\varnu, \vvarkappa) \vchi_{i \varnu,  \vvarkappa          } (z) \nd \mu (\varnu, \vvarkappa) , 
 \end{align}
 \begin{align}\label{6eq: Phi}
 	\Upsilon_{\rho}^{\pm} (z) =  \frac{ X^4 \log T}{ M   T } \viint_{ \EuScript{A}_{\vepsilon} (T/M+X ,  \sqrt{\rho} T  ) } \vlambda_{\rho}^{\pm} (\varnu, \vvarkappa) \vchi_{i \varnu,  \vvarkappa          } (z) \nd \mu (\varnu, \vvarkappa) , 
 \end{align}
for   $ \vlambda_{\rho}^{\pm} (\varnu, \vvarkappa) \Lt 1 $ {\rm(}of suppressed $X$-inert variables $u_1, u_2${\rm)}. 

\begin{remark}
	Actually, with some efforts, one may remove $\widehat{\varww}{}^{+}_{0 }  (u_1, u_2 ; \varLambda ) $ from   {\eqref{6eq: w = w rho}} and,  for $\rho$ small,   improve the fractions in {\rm\eqref{6eq: Phi o}} and {\rm\eqref{6eq: Phi}}. 
\delete{	\begin{align}\label{6eq: w = w rho, 2} 
		\widehat{\varww}{}^{\natural}  (u_1, u_2 ; \varLambda ) = \left\{  \begin{aligned} 
			& \widehat{\varww}{}^{+}_{0 }  (u_1, u_2 ; \varLambda )   +  O (T^{-A}) , & &  \text{ if  } \sqrt{\rho} \Lt  {X T^{\vepsilon}}/ {T },   \\
			&  \widehat{\varww}{}^{+}_{\rho}  (u_1, u_2 ; \varLambda ) + \widehat{\varww}{}^{-}_{\rho}  (u_1, u_2 ; \varLambda ) +  O (T^{-A}), & & \text{ if  } \sqrt{\rho} \Gt {X T^{\vepsilon}}/ {T },   
		\end{aligned} \right.
	\end{align}  }
\end{remark}


\end{lem}

	 
As $ \vlambda_{\rho}^{\pm} (\varnu, \vvarkappa)$ is implicitly dependent on $X$-inert $u_1$, $u_2$. 	 By Lemma \ref{lem: w(z) Mellin},  we may separate $u_1$, $u_2$ as follows{\rm:} 
	 \begin{align}\label{6eq: inert Mellin}
	 \vlambda_{\rho}^{\pm} (\varnu, \vvarkappa; \bfu) = 	\viint \hskip -2pt  \viint_{\boldsymbol{\EuScript{A}}_{\vepsilon} (X) }  \xi_{\rho}^{\pm} (\varnu, \vvarkappa ; \bfnu, \bfkappa)\overline{\vchi_{i\bfnu, \bfkappa }(\bfu) }\nd \mu(\bfnu ,\bfkappa) + O (T^{-A}), 
	 \end{align}
 where   $\xi_{\rho}^{\pm} $ is still bounded, and  in the standard multi-variable notation $ \bfu = (u_1, u_2) $, $ \bfnu = (\varnu_1, \varnu_2) $, $\bfkappa = (\vvarkappa_1, \vvarkappa_2)$, $ \boldsymbol{\EuScript{A}}_{\vepsilon} (X) =  {\EuScript{A}}_{\vepsilon} (X) \times  {\EuScript{A}}_{\vepsilon} (X)$ {\rm(}as in Definition \ref{defn: square}{\rm)}, and $\vchi_{i\bfnu, \bfkappa }(\bfu) = \vchi_{i {\varnu_1},  {\vvarkappa_1} }( {u_1})  \vchi_{i {\varnu_2},  {\vvarkappa_2} }( {u_2})  $. 
 
\begin{remark}\label{rem: nu integral}
	Note that the $(\bfnu, \bfkappa)$-integral is practically harmless, as it yields only an extra factor $O (X^4 T^{\vepsilon})$, while $\varnu_1$, $\varnu_2$ or $\vvarkappa_1, \vvarkappa_2$ may be absorbed into $\varnu$ or $\vvarkappa$ by  $$ {\EuScript{A}}_{\vepsilon} (X) \subset \EuScript{A}_{\vepsilon} \big( T/M+X ,  \sqrt{\rho} T   \big) , $$
for $\sqrt{\rho} \Gt {X T^{\vepsilon}}/ {T }$. 
\end{remark}

\subsection{Proof of Lemma \ref{lem: w (0, 0)}} 
By the change of variable $z_2 \rightarrow z/ z_1 $, we may write
\begin{align*} 
	\varvv (0,0;\zeta) =  \widehat{\varvv} (\zeta) = \viint        \varvv (z) e  [  \zeta z   ]  \nd z , \qquad \varvv (z) = \viint  \varww (z_1)  { \overline{\varww}  ( z/z_1 )  }    \frac { \nd z_1} {{|z_1|^2}} .
\end{align*}
Clearly $ \varvv \in C_c^{\infty}[1/2, 2]$ is still $X$-inert and hence $ \widehat{\varvv} (\zeta)  $ is negligibly small if $ |\zeta | > X T^{\vepsilon} $. Moreover, 
\begin{align*}
	\frac{\partial^{i+j} \widehat{\varvv} (\zeta) }{\partial \zeta^i \partial \widebar{\zeta}^j } \Lt \lp 1 + \frac{|\zeta|}{X} \rp^{-A} .
\end{align*} Therefore
\begin{align}\label{6eq: integral w (0, 0)}
	\widehat{\varww} (0, 0 ; \varLambda ) = M T^2 \viint_{ \widehat{\EuScript{A}} } g (Mr) e (2Tr/\pi) \widehat{\varvv} (\varLambda \psi (r, \omega))    \nd \widehat{\mu} (r, \omega),
\end{align} 
and the integral is bounded by the area of the region given by
\begin{align*} 
		|\psi ( r, \omega )| \leqslant \frac {X T^{\vepsilon}  } {|\varLambda|}  . 
\end{align*}
Recall from    \eqref{3eq: trh} and \eqref{3eq: defn of psi} that 
\begin{align*}
\psi (r, \omega) = 2(	\cosh r \cos \omega - 1 + i \sinh r \sin \omega) , 
\end{align*} 
 hence 
\begin{align*}
	|\psi (r, \omega)| = 2 (\cosh r - \cos \omega) = 4  ( \sinh^2 (r/2) + \sin^2 (\omega/2)  ) ,  
\end{align*} 
and the bound in \eqref{6eq: bound w(0, 0)} follows immediately. Further, on the above region 
$$ \frac{\partial \psi(r,\omega) }{\partial r} \Lt  \bigg({\frac{XT^\vepsilon}{|\varLambda|} }\bigg)^{1/2} , $$ 
so 
$$  \frac{\partial^{j } \widehat{\varvv} (\varLambda \psi (r, \omega) ) }{\partial r^j  } \Lt  ( {X|\varLambda| T^{\vepsilon}})^{j / 2} ;  $$ 
 this may be seen by a variant of the  Fa\`a di Bruno formula \cite{Faa-di-Bruno}. Since the $r$-integral in \eqref{6eq: integral w (0, 0)} is   Fourier, we infer that  $	\widehat{\varww} (0, 0 ; \varLambda )$ is negligibly small as long as $ \sqrt{X |\varLambda |} \leqslant T^{1-\vepsilon} $ (and $M \leqslant T^{1-\vepsilon}$).

\subsection{Proof of Lemma \ref{lem: v (u; theta)}}\label{sec: Fourier}

By the change of variable $z_2 \rightarrow z/ z_1 $, we have
\begin{align}\label{6q: 2nd integral}
	\varvv (u_1,u_2;\zeta) = \viint \varvv_{\natural} (z;u_1, u_2)  e  [  \zeta z   ]  \nd z,
\end{align}
with
\begin{align}\label{6q: 1st integral}
	\varvv_{\natural} (z;u_1, u_2) = \viint   \frac{\varww (z_1) \overline{\varww}  ( z/z_1 )  }{|z_1|^2}  e \bigg[  u_1 z_1 + \frac{u_2 z}{z_1} \bigg]  \nd z_1 . 
\end{align}

As  $|u_1| + | u_2| > X  T^{\vepsilon}$, by applying Corollary \ref{cor: staionary phase, complex} with $P = 1/X$, $Q = 1$, $R =   |u_1| + |u_2|  $, and $Z = |u_2|$, we infer that the integral $ 	\varvv_{\natural} (z;u_1, u_2) $ as in \eqref{6q: 1st integral}  is negligibly small unless $ u_1/u_2 \in A [1/2, 2] $, say. Next, for such $u_1$ and $u_2$, on the change  $z_1\rightarrow z_1\sqrt{u_2 z/u_1}$, we have 
$$ \varvv_{\natural} (z;u_1, u_2) = \viint \varww_{\natural} (z, z_1)  e \bigg[  \sqrt{u_1 u_2 z}\bigg(  z_1 + \frac{1}{z_1}  \bigg)   \bigg]  \nd z_1, $$
where $ \varww_{\natural} \in C_c^{\infty} (  A[1, 2] \times A[1/ {2}, 2 ]  )$ is an $X$-inert weight function (in the sense of \eqref{2eq: inert annulus, 2}). For brevity, the dependence of $ \varww_{\natural} $ on the $X$-inert variables $\bfu = (u_1, u_2)$ has been suppressed.  
In order to localize the stationary points at $\pm 1$,  we introduce a suitable (inert)  partition of unity so that
$$ \varvv_{\natural} (z;u_1, u_2) = \sum_{\pm}  \viint \varww_{\flat} (z,   z_1 )  e \hskip -1.5pt   \left[ \pm \sqrt{u_1 u_2 z}\lp z_1 + \frac{1}{z_1}  \rp \right]  \nd z_1 + O (T^{-A}), $$ 
where $  \varww_{\flat} \in C_c^{\infty} (A[1, 2] \times  D_1 [1 / {2} ]   ) $ is $X$-inert (the disc is centered at $1$); again,  the error is negligible by   Corollary  \ref{cor: staionary phase, complex} with $P = 1/X$, $Q = 1$, and $R = Z = \sqrt{ |u_1 u_2|}$. Now we introduce the new variable $ v $ defined by 
\begin{align*}
	 z_1 + \frac{1}{z_1} = v ^2 + 2,\qquad v = \sqrt{ z_1} - \frac{1}{\sqrt{z_1}},
\end{align*}
where   $\sqrt{\phantom{z} }$  is chosen to be the principal branch. It follows that 
\begin{align*}
	\varvv_{\natural} (z;u_1, u_2) = \sum_{\pm} e [\pm 2 \sqrt{u_1 u_2 z}] \viint \varvv_{\flat} (z,  v)  e   \hskip -1.5pt    \left[ \pm \sqrt{u_1 u_2 z} v^2   \right]  \nd v + O (T^{-A}),
\end{align*}
where $ \varvv_{\flat} \in C_c^{\infty} ( A[1, 2] \times D  [1/\sqrt{2} ] )  $ is $X$-inert (in the sense of \eqref{2eq: inert annulus} and \eqref{2eq: inert disc}).  Thus we conclude by Lemma \ref{lem: Weil} with $a =  1 /\sqrt{2}$, $S = 1$, and $ b =Y = X $ that 
\begin{align}\label{6eq: v, 2}
	 	\varvv_{\natural} (z;u_1, u_2) = \sum_{\pm}  \frac {e [\pm 2 \sqrt{u_1 u_2 z}]} {{\sqrt{|u_1 u_2 z|}} }  \breve{\varvv}_{\flat} (z,  \pm \sqrt{u_1 u_2 z}) + O (T^{-A}), 
\end{align}
where $ \breve{\varvv}_{\flat} (z,  u) \in C^{\infty} (A[1,2] \times A[ X T^{\vepsilon},  \infty) ) $ (compactly supported in $z$ of course)  is an $ X^2 $-bounded and $X$-inert function. 

On inserting \eqref{6eq: v, 2} into \eqref{6q: 2nd integral} and making the substitution $ \pm \sqrt{z} \ra z $, we obtain  
\begin{align}\label{6eq: integral v, 2}
	\varvv (u_1,u_2;\zeta) =  \frac 1 {{\sqrt{|u_1 u_2|}} }  \viint \varww  (z)     e  [  \zeta z^2 + 2 \sqrt{u_1 u_2 } z  ]      \nd z + O (T^{-A}),
\end{align} 
with 
\begin{align*}
\varww  (z) = 	4|z| \breve{\varvv}_{\flat} (z^2, \sqrt{u_1 u_2 } z) , 
\end{align*}
where $ \varww  \in C_c^{\infty} (A[1, \sqrt{2}]) $ is $X^2$-bounded and $X$-inert, and  for brevity $\bfu  $ (also $X$-inert) has been suppressed again. By applying Corollary \ref{cor: staionary phase, complex} with  $P = 1/X$, $Q = 1$, $R = |\zeta| + \sqrt{|u_1 u_2|}$, and $Z = |\zeta|$, we infer that the integral $\varvv (u_1,u_2;\zeta)$ is negligibly small unless  $ u_1 / \zeta , u_2 / \zeta \in A[\sqrt[4]{2}/2 , 2 \sqrt[4]{2} ]  $, say. 
Then, for such $u_1$, $u_2$, and $\zeta$,  we introduce
\begin{align*}
	w = z +  \frac {\sqrt{u_1 u_2}} {\zeta}    
\end{align*}
to transform \eqref{6eq: integral v, 2} into 
\begin{align*}
	\varvv (u_1,u_2;\zeta) = \frac 1 {{\sqrt{|u_1 u_2|}} }  e \hskip -1pt  \left[ - \frac {u_1 u_2} {\zeta } \right]  \viint \varvv (w)   e  \big[  \zeta \shskip w^2   \big]      \nd w + O (T^{-A}), 
\end{align*}  
where $ \varvv \in C_c^{\infty} (D [\sqrt{2} + 2 \sqrt[4]{2} ])$ is $ X^2 $-bounded and $X$-inert (with implicit dependence on the $X$-inert variables  $\bfu $ and $\zeta$). Finally, we deduce the expression \eqref{6eq: integral v} by Lemma \ref{lem: Weil} with  $a = \sqrt{2} + 2 \sqrt[4]{2} $, $S = X^2$, and $ b =Y = X $. Note that  we have absorbed $\sqrt{|u_1 u_2|} / |\zeta| $ into 
$\breve{\varvv} (   \zeta) $ so that the denominator reads $ |u_1 u_2|$ instead of $\sqrt{|u_1u_2|} |\zeta|$.  
\subsection{Proof of Lemma \ref{lem: w (u; Lambda)}}\label{sec: Mellin}

In view of \eqref{6eq: integral w}, \eqref{6eq: defn wn}, and \eqref{6eq: integral v}, we may write 
\begin{align*}
	 \widehat{\varww}{}^{\natural}  (u_1, u_2 ; \varLambda ) = \frac{M T^2} {|u_1 u_2 |  } \Upsilon \Big( \frac {u_1 u_2} {4 \varLambda} \Big) + O (T^{-A}),  
\end{align*} 
\begin{align*}
\Upsilon  ( z ) =   \viint_{ \widehat{\EuScript{A}} } e  (  {2Tr} /{\pi}  ) e [ -  z  \tau (r, \omega) ]  g (Mr)  V(r, \omega)  \nd \widehat{\mu} (r, \omega) , 
\end{align*}
 with  
\begin{align*}
 	\tau (r, \omega) =    1 + \frac 4 {\psi (r, \omega)},  \qquad V (r, \omega) = \breve{\varvv} (\varLambda \psi (r, \omega)), 
\end{align*}  
where  $\breve{\varvv} (\zeta)$ is $X^4$-bounded and $X$-inert  (in the sense of \eqref{2eq: inert annulus}).   
It follows from  \eqref{3eq: trh} and \eqref{3eq: defn of psi} that 
\begin{align*}
	|\psi (r, \omega)| = 2 (\cosh r - \cos \omega) = 4  ( \sinh^2 (r/2) + \sin^2 (\omega/2)  ) ,  
\end{align*} 
and that,  in the polar coordinates, if we write $  \tau (r, \omega) = \rho    (r,\omega) e^{i \theta(r,\omega)}$, then  
\begin{equation*}
	\rho(r,\omega)=\frac{\cosh r + \cos \omega}{ \cosh r - \cos \omega }, 
\end{equation*} 
\begin{equation*}
	\cos \theta(r,\omega) =    \frac{\sinh^2 r - \sin^2 \omega}{\sinh^2 r + \sin^2 \omega}, \qquad 	\sin \theta(r,\omega) =  - \frac{ 2 \sinh r\sin \omega}{\sinh^2 r + \sin^2 \omega} , 
\end{equation*} 
Moreover,   the Schwartz $g (Mr)$ is negligibly small if $|r| > M^{\vepsilon} / M$, while, by Lemma   \ref{lem: v (u; theta)}, so is   $V (r, \omega)$ if $ r^2 + 4\sin^2 (\omega/2) <  X T^{\vepsilon}/ |\varLambda | $.  

Next we introduce a smooth dyadic partition for the  integral $\Upsilon (z)$ according to the value of $r^2 + 4\cos^2 (\omega/2)$ or $ r^2 + 4\sin^2 (\omega/2) $ (note the simple identity $\sin \omega = 2 \cos (\omega/2)  \sin (\omega/2)$). Set   $ \phi_{0} =   {XT^{\vepsilon}} / {|\varLambda|} $. It follows that 
\begin{align}\label{6eq: w = w phi} 
	\widehat{\varww}{}^{\natural}  (u_1, u_2 ; \varLambda ) = \widehat{\varww}{}^{+}_{0 }  (u_1, u_2 ; \varLambda ) + \sum_{\pm} \sum_{ \phi_{0}    \Lt \shskip  \phi \shskip \Lt \shskip    1 } \widehat{\varww}{}^{\pm}_{\phi}  (u_1, u_2 ; \varLambda ) +  O (T^{-A}) , 
\end{align} for  dyadic $\phi = 2^{-k/2}$ {\rm(}$k \geqslant -1${\rm)},  where, with abuse of notation (however, it will be shown later that only those $\phi \asymp \rho$ contribute, so $\phi$ is used instead of $\rho$ in a large part of the proof),  
 \begin{align}\label{6eq: w = Phi, 2}
	\widehat{\varww}{}^{\pm}_{\phi}   (u_1, u_2 ; \varLambda ) = \frac{M T^2} {|u_1 u_2 |  }  \Upsilon_{\phi}^{\pm}  \Big( \frac {u_1 u_2} {4 \varLambda} \Big) , 
\end{align}   
\begin{align}\label{6eq: Upsilon}
	\Upsilon_{\phi}^{\pm}  ( z ) =   \viint  e  (  {2Tr} /{\pi}  ) e [ -   z  \tau (r, \omega) ]  g (Mr)  V_{\phi}^{\pm} (r, \omega)  \nd \widehat{\mu} (r, \omega) , 
\end{align}
where $V_{0}^{+}  (r, \omega) $ and $V_{\phi}^{\pm} (r, \omega) $ respectively are supported on  the regions  defined by 
\begin{align} 
   r^2 + 4 \cos^2  (\omega/2) \leqslant  \phi_{0}  , \qquad   \left\{\begin{aligned}
  &    r^2 +  4 \cos^2  (\omega/2)    \sim \phi , & & \text{ if } +,  \\
   	&    r^2 +  4 \sin^2  (\omega/2)    \sim \phi , & & \text{ if } - , 
   \end{aligned}  \right.  
\end{align}   
with bounds 
\begin{align}\label{6eq: bounds for V phi}
\frac {\partial^{i+j}  V_{\phi}^+ (r, \omega)} {\partial r^i \partial \omega^j}  \Lt X^{4 + i +j } ,   \qquad 	\frac {\partial^{i+j}  V_{\phi}^{-} (r, \omega)} {\partial r^i \partial \omega^j}  \Lt X^{4 } \bigg( \frac X  {\sqrt{\phi}} \bigg)^{i+j} . 
\end{align}
For the second bound in \eqref{6eq: bounds for V phi}, one needs to  examine the powers of $\sinh r$ and $\sin \omega$ in the derivatives of $  \breve{\varvv} (\varLambda \psi (r, \omega))$ (a variant of the  Fa\`a di Bruno formula would be helpful).   
Since on the above region $|\psi (r, \omega)| \asymp 1 $ or $ \phi$ in the $+$ or $-$ case, by Lemma   \ref{lem: v (u; theta)} (the discussion below \eqref{6eq: integral v}) we infer that $ \widehat{\varww}{}^{\pm}_{\phi}   (u_1, u_2 ; \varLambda ) $ is negligibly small unless 
\begin{align}\label{6eq: supp} 
		|u_1| , |u_2| \asymp \left\{ \begin{aligned} & |\varLambda | , & & \text{ if } +, \\
			&   |\phi \varLambda |  , & & \text{ if } -. 
		\end{aligned} \right. 
\end{align} 

\delete{
\begin{itemize}
	\item [(1)] if  $ \sqrt{\rho} \Lt  {X T^{\vepsilon}}/ {T } $ then all the $\Upsilon_{\phi}^{\pm}  ( z )$ with $\phi \neq 0$ are negligibly small, and 
	\item [(2)]   if  $ \sqrt{\rho} \Gt  {X T^{\vepsilon}}/ {T } $ then $\Upsilon_{\phi}^{\pm}  ( z )$ is not negligibly small only when $\phi \asymp \rho$. 
 \end{itemize}
} 

First let us treat the easier $ \Upsilon_{0}^+ (z)$ by Lemma  \ref{lem: w(z) Mellin}.  Now $e [ -  z  \tau (r, \omega) ]$ is $XT^{\vepsilon}$-inert since $ |z| \asymp |\varLambda| $ by \eqref{6eq: supp} while $|\tau (r, \omega)| \Lt \phi_0  $ on the support of $V_0 (r, \omega)$, but $ |\phi_0 \varLambda|  = X T^{\vepsilon}$.  It follows from  the Mellin expression in \eqref{2eq: Mellin of w(z)} for $\varww (z) = e [ -  z  \tau (r, \omega) ]$ that 
\begin{align}
	\Upsilon_{0}^{+}  (z) = \viint_{ 	\EuScript{A}_{\vepsilon} (X )  }   {\xi    (2\varnu, \vvarkappa)} I_{0}^{+} (2\varnu, \vvarkappa) \overline{\vchi_{i\varnu, \vvarkappa}(z)}  \nd \mu(\varnu, \vvarkappa) + O (T^{-A}),
\end{align}
where $ \xi (\varnu, \vvarkappa)   $ is bounded,  $ \EuScript{A}_{\vepsilon} (X) $ is from Definition \ref{defn: square}, and
\delete{\begin{align}\label{6eq: I0 (nu, m)}
	I_{0}^{+} ( \varnu, \vvarkappa) =   	\viint 
	g (Mr) V_{0}^{+} (r,\omega) \exp  ( i f_T( r, \omega; \varnu, \vvarkappa) )   \nd \widehat{\mu}(r,\omega) ,
\end{align} 
\begin{align}
f_T( r, \omega; \varnu, \vvarkappa) = {4Tr}   + {\varnu} \log \rho(r,\omega) +   {\vvarkappa} \theta(r,\omega). 
\end{align} } 
\begin{align*} 
	I_{0}^{+} ( \varnu, \vvarkappa) \Lt   \viint 
	\big|   V_{0}^{+} (r,\omega) \big|    \nd \widehat{\mu}(r,\omega) \Lt   {X^4   \phi_0  } =  \frac {X^5 T^{\vepsilon}}  {|\varLambda|} . 
\end{align*} 
Then follows the formula \eqref{6eq: Phi o} for $  |T/ \varLambda| \Lt \sqrt{\rho} $.

Now it is left to consider $ \Upsilon_{\phi}^{\pm} (z)$.   By  Lemma   \ref{lem: v (u; theta)} (or \eqref{6eq: supp}),  we have   $$  |z \tau (r, \omega) | \asymp |\varLambda \psi (r, \omega)^2  \tau (r, \omega)| \asymp  |\phi \varLambda |  . $$   
By the Mellin expression for $e [ -  z  \tau (r, \omega) ]$ as in \eqref{2eq: Mellin} in Lemma \ref{lem: e[z] Mellin}, we reformulate \eqref{6eq: Upsilon} into
\begin{align}
	\Upsilon_{\phi}^{\pm}  (z) = \viint_{ 	\EuScript{A}  (|\phi  \varLambda | )  }  \overline{\xi_{|\phi \varLambda |   }  (2\varnu, \vvarkappa)} I_{\phi}^{\pm} (2\varnu, \vvarkappa) {\vchi_{i\varnu, \vvarkappa}(z)}  \nd \mu(\varnu, \vvarkappa) + O (T^{-A}),
\end{align}
where $ \EuScript{A} (|\phi \varLambda | ) $ is from Definition  \ref{defn: square},   
$$ \xi_{|\phi \varLambda  | }  ( \varnu, \vvarkappa) = O \bigg(\frac {\log T} { |\phi \varLambda   |} \bigg), $$ 
and  
\begin{align}\label{6eq: I (nu, m)}
	I_{\phi}^{\pm} ( \varnu, \vvarkappa) =   	\viint 
	g (Mr) V_{\phi}^{\pm} (r,\omega) \exp  ( i f_T( r, \omega; \varnu, \vvarkappa)  )   \nd \widehat{\mu}(r,\omega) ,
\end{align} 
\begin{align}
	f_T( r, \omega; \varnu, \vvarkappa) = {4Tr}   + {\varnu} \log \rho(r,\omega) +   {\vvarkappa} \theta(r,\omega). 
\end{align} 
By trivial estimation, 
\begin{align*}
	I_{\phi}^{\pm} ( \varnu, \vvarkappa) = O \bigg( \frac {X^4   \sqrt{\phi}} {M}  \bigg). 
\end{align*}
Thus on re-scaling $ \overline{\xi_{|\phi \varLambda  | } (2\varnu, \vvarkappa)} I_{\phi}^{\pm} (2\varnu, \vvarkappa)$,  we get the expression 
\begin{align}\label{6eq: Phi, 2}
	\Upsilon_{\phi}^{\pm} (z) =  \frac{ X^4 \log T}{ M   \sqrt{\phi} |\varLambda|  }   \viint_{ \EuScript{A} (|\phi  \varLambda|)  } \vlambda_{\phi}^{\pm} (\varnu, \vvarkappa) \vchi_{i \varnu,  \vvarkappa          } (z) \nd \mu (\varnu, \vvarkappa) .  
\end{align}
However, this is not yet close to \eqref{6eq: Phi}. Further, we claim  that $ I_{\phi}^{\pm} (\varnu, \vvarkappa) $ is negligibly small unless $ (\varnu, \vvarkappa) $ is on the annulus $ \EuScript{A} (\sqrt{\phi} T) $. Thus in order for $  \Upsilon_{\phi}^{\pm} (z) $ not to be negligibly small one must have $ |\phi \varLambda| \asymp \sqrt{\phi} T $ or in other words $ \phi \asymp \rho$ as $  |T/\varLambda|^2 \sim  \rho $. Consequently, according as  $\sqrt{\rho} \Gt {X T^{\vepsilon}}/ {T }$ or not,  \eqref{6eq: w = w phi}   turns   into \eqref{6eq: w = w rho} or \eqref{6eq: w = w rho, 1}  if we let $ \widehat{\varww}{}^{\pm}_{\rho}  (u_1, u_2 ; \varLambda ) $ be the sum of those surviving $ \widehat{\varww}{}^{\pm}_{\phi}  (u_1, u_2 ; \varLambda )$. 




To prove the claim, let us first analyze the phase $ f_{T}  ( r, \omega) = f_{T}  ( r, \omega; \varnu, \vvarkappa) $. 
Note that if we set 
$$ A (r, \omega) = \frac{\partial \log \rho(r,\omega) }{\partial r} = \frac{\partial \theta(r,\omega) }{\partial \omega} = - \frac{2 \sinh r \cos \omega}{\sinh^2 r + \sin^2 \omega} ,$$
$$ B (r, \omega) = \frac{\partial \log \rho(r,\omega) }{\partial \omega} = - \frac{\partial \theta(r,\omega) }{\partial r} = - \frac{2 \cosh r \sin \omega}{\sinh^2 r + \sin^2 \omega}, $$
then
\begin{align*}
	\partial f_{T}  ( r, \omega) / \partial r = 4 T + \varnu A (r, \omega) - \vvarkappa B (r, \omega), \quad  \partial f_T (r, \omega) / \partial \omega =   \varnu B (r, \omega) + \vvarkappa A (r, \omega). 
\end{align*}
Consequently, 
\begin{align*}
(	\partial f_{T}  ( r, \omega) / \partial r)^2 + (	\partial f_{T}  ( r, \omega) / \partial \omega)^2 \geq \big( 4T - \sqrt{ \varnu^2 + \vvarkappa^2} \sqrt{A ( r, \omega)^2 + B ( r, \omega)^2} \big)^2 ,
\end{align*}
while 
$$ A ( r, \omega)^2 + B ( r, \omega)^2 = \frac{4}{\sinh^2 r + \sin^2 \omega}. $$
Moreover
\begin{align*}
	 \frac {\partial^{i+j} A   (r, \omega) }  {\partial r^i \partial \omega^j }, \ \frac {\partial^{i+j} B   (r, \omega) }  {\partial r^i \partial \omega^j } \Lt \bigg( \frac {1} {\sqrt{ \sinh^2 r + \sin^2 \omega}} \bigg)^{i+j}. 
\end{align*}
Note that $ \sinh^2 r + \sin^2 \omega \asymp \phi $ on the support of $ V_{\phi}^{\pm} (r,\omega)$.  By Lemma \ref{lem: staionary phase, dim 2, 2} with $P = \min \{ 1/M, \varUpsilon \}$, $ \varUpsilon = 1/X$ or $\sqrt{\phi} /X$ (see \eqref{6eq: bounds for V phi}),  $ Z = \sqrt{\varnu^2 + \vvarkappa^2},\, Q =\varPhi = \sqrt{\phi} $, and $R= T + \sqrt{\varnu^2 + \vvarkappa^2} / \sqrt{\phi} $, we deduce for $  \sqrt{ \upnu^2 + \upkappa^2 } \asymp |\phi \varLambda| \Gt X T^{\vepsilon}$ that $I_{\phi}^{\pm} (\varnu, \vvarkappa)$ is negligibly small unless $ \sqrt{\varnu^2 + \vvarkappa^2} \asymp T\sqrt{\phi}$ as claimed. 

At this point, we have arrived at:
\begin{align}\label{6eq: Phi, 3} 
	\Upsilon_{\phi}^{\pm} (z) =  \frac{ X^4 \log T}{ M   T  }   \viint_{ \EuScript{A} (\sqrt{\phi} T  )    } \vlambda_{\phi}^{\pm} (\varnu, \vvarkappa) \vchi_{i \varnu,  \vvarkappa          } (z) \nd \mu (\varnu, \vvarkappa) . 
\end{align} 
It will be simpler to extend the annulus $ \EuScript{A} (\sqrt{\phi} T  )  $ into the square $ \EuScript{A}_{\vepsilon} (\sqrt{\phi} T  )  $ (see Definition \ref{defn: square}). Our final step is to compress the square $ \EuScript{A}_{\vepsilon} (\sqrt{\phi} T  )  $ into the rectangle  $ \EuScript{A}_{\vepsilon} (T/M+X, \sqrt{\phi} T  )  $ (see Definition \ref{defn: rectangle}).


Note that the domain is compressed only if $ \sqrt{\phi} > T^{\vepsilon}/ M$ (to be strict, the first $T^{\vepsilon}$ in \eqref{6eq: defn A(V, U)} is considered to be $T^{2\vepsilon}$), so   we   need to prove for such $\phi$ that (the $\omega$-integral in) $I_{\phi}^{\pm} ( \varnu, \vvarkappa) $ is negligibly small if $ |\varnu |  > (T/M+X) T^{2\vepsilon} $. For $ \sqrt{\phi} > T^{\vepsilon}/ M$, the domain of the integral  in \eqref{6eq: I (nu, m)} may be restricted to
\begin{align*}
	| r | \leqslant M^{\vepsilon} / M, \qquad \left\{\begin{aligned}
	&	   |2 \cos (\omega/2) |  \sim \sqrt{\phi/2} , & & \text{ if $+$,}   \\
		&   |2 \sin (\omega/2) |  \sim \sqrt{\phi/2} ,  & & \text{ if $-$.}  
	\end{aligned} \right.
\end{align*}  
Recall that
\begin{align*}
 B (r, \omega) =	\frac {\partial \log \rho (r, \omega) }  {\partial \omega} = - \frac{ 2 \cosh r \sin \omega  }{  \sinh^2 r + \sin^2 \omega }, \quad A  (r, \omega) = \frac {\partial   \theta (r, \omega) }  {\partial \omega} = - \frac{ 2 \sinh r \cos \omega  }{  \sinh^2 r + \sin^2 \omega }, 
\end{align*}
so on this domain 
\begin{align*}
\big|\partial f_T (r, \omega) / \partial \omega\big| = 	\big| \varnu B (r, \omega) + \vvarkappa A (r, \omega) \big| \Gt \frac {|\varnu| } {\sqrt{\phi}} - \frac {|\vvarkappa| M^{\vepsilon}} { \phi M  } ,
\end{align*}
and  
\begin{align*}
	\frac {\partial^{j} B (r, \omega) }  {\partial \omega^j } \Lt \bigg(\frac 1 {\sqrt{\phi}} \bigg)^{j+1}, \qquad \frac {\partial^{j}   A (r, \omega) }  {\partial \omega^{j}} \Lt \frac {M^{\vepsilon}} {M \sqrt{\phi}}  \bigg(\frac 1 {\sqrt{\phi}} \bigg)^{j+1} . 
\end{align*}
For $(T/M+X  ) T^{2\vepsilon}   < |\varnu| \leqslant  \sqrt{\phi}  T^{1+\vepsilon}  $ and $|\vvarkappa|  \leqslant \sqrt{\phi}  T^{1+\vepsilon}    $, on applying Lemma \ref{lem: staionary phase, dim 1, 2} to the $\omega$-integral with $P = 1/X$ or $\sqrt{\phi}/X$ (see  \eqref{6eq: bounds for V phi}),  $ Q = \sqrt{\phi}$, $R = |\varnu| / \sqrt{\phi}$,  and $ Z = |\varnu| + |\vvarkappa| M^{\vepsilon}/ M \sqrt{\phi}  $,  we infer that the integral $I_{\phi}^+ (\varnu, \vvarkappa)$ is indeed negligibly small. 

\begin{remark}\label{rem: worse than trivial}
	Since the integral domain of $I_{\phi}^{\pm} (\varnu, \vvarkappa)$ is already very  slim in the case when $ {\phi} \asymp 1 $ and $M = T^{1-\vepsilon}$, the stationary phase analysis as in \cite[\S 13.2]{Qi-GL(3)} has no advantage over trivial estimation.  Actually, it is worse particularly when $ r^2 + \cos^2 \omega $ takes   small value on the integral domain. 
\end{remark}

\section{Proof of Proposition \ref{prop: initial reduction}}  
 
Recall that  $ 1 \Lt N < X T^{ 2+\vepsilon}$ and $ 1 \Lt C \Lt (N/T)^2 $,  and that our goal is to prove that the sum $S (N, C)$ as in  \eqref{4eq: S(N, C), Poisson} after Poisson has bound  $O (X^{12}   T^{4+\vepsilon}/ M )$.

\subsection{Contribution from the Zero Frequency}  

By \eqref{5eq: T(0,n;c)} in Lemma \ref{lem: exp sum}, 
\begin{align*}
	T (0, 0; c) & =  \left\{ \begin{aligned}
		&	\RN (c )^{3/2} \varphi (c), & & \text{ if } (c) \text{ is square, }      \\
		& 0 ,  & & \text{ if otherwise.}
	\end{aligned} \right.
\end{align*}
By Lemma \ref{lem: w (0, 0)} we have uniformly
\begin{align*}
	\widehat{\varww} (0, 0 ; \varLambda ) \Lt  X^2  M T^{ \vepsilon}  . 
\end{align*}
Thus the contribution from the zero frequency   ($n_1=n_2=0$)  is bounded by 
\begin{align*}
	 N \sum_{\RN(\frc ) \sim \sqrt{C}} \frac {1 } { \RN (\frc )^{6  } } \cdot \RN (\frc )^{5} \cdot X^2 M T^{\vepsilon} \Lt X^2 N M T^{\vepsilon} \Lt X^3 M T^{2+\vepsilon},
\end{align*}
as desired. 

\subsection{Contribution from Small Frequencies} 

Next we estimate trivially the contribution from those small frequencies $ \RN (n_1)  ,  \RN (n_2) \Lt X^2 C  T^{\vepsilon} / {N} $ (zero frequency $n_1=n_2=0$ excluded) which are not covered by Lemma \ref{lem: w (u; Lambda)}. 

Recall from  \eqref{5eq: non-vanishing}, \eqref{5eq: explicit T(n; c)}, and \eqref{6eq: trivial bound w} that
\begin{align*}
	  T (n_1, n_2; c) & \Lt \left\{ \begin{aligned}
		&	 \RN (c)^{3/2}  \RN (d), & & \text{ if } (d) = (n_1, c) = (n_2, c) , \\
		& 0, & & \text{ if otherwise, } 
	\end{aligned}  \right.  
\end{align*}
\begin{align*}
	\widehat{\varww} (u_1, u_2 ; \varLambda ) \Lt T^{2 }.  
\end{align*} 
In view of \eqref{4eq: S(N, C), Poisson}, it follows that such contribution 
is bounded by 
\begin{align*}  
	    N     \sum_{\RN(\frc) \sim C} \frac {1 } { \RN (\frc)^{3  } }     \sum_{\mathfrak{d} | \frc }   \RN (\frc)^{3/2}   \RN (\mathfrak{d})   T^{2}    \frac { X^2 C T^{\vepsilon} } { N \RN (\mathfrak{d})  } \bigg( 1+ \frac { X^2 C T^{\vepsilon} } { N \RN (\mathfrak{d})  } \bigg)   \Lt X^2 \sqrt{C}  T^{2+\vepsilon} \bigg( 1 +  \frac{X^2 C   }{N}  \bigg)    ,   
\end{align*} 
and hence $O ( X^6   T^{3+\vepsilon})$ due to $ C \Lt (N/T)^2$ and $N < X T^{2+\vepsilon}$.

	\subsection{Partition and Transformation}  By  the partition  \eqref{6eq: w = w rho} or \eqref{6eq: w = w rho, 1}  in Lemma \ref{lem: w (u; Lambda)}, we partition  $S (N, C) $ (as in   \eqref{4eq: S(N, C), Poisson}) accordingly and consider each partial sum
	\begin{align}  
		S^{\pm}_{\rho} (N, C)    =   \hskip -1pt  \frac {N } {|d_F|} \hskip -1pt \sum_{ 16 \shskip | \shskip c \, : \, \RN (c) \sim C }  \hskip -1pt \frac 1 {\RN (c)^3}  \hskip -1pt \mathop{\sum \sum}_{n_1, \, n_2}  T^{\natural} (n_1, n_2; c)  \widehat{\varww}^{\pm}_{\rho}  \hskip -1pt \bigg( \hskip -1pt \frac{n_1 \sqrt{N}}{\updelta_F c}, \frac{n_2 \sqrt{N}}{\updelta_F c} ; \frac {N } { \updelta_F c}  \hskip -1pt \bigg)  \hskip -1pt , 
	\end{align}
where $T^{\natural}$ and $ \widehat{\varww}^{\pm}_{\rho} $ are given by \eqref{5eq: Tn(n; c)}, \eqref{6eq: defn wn}, \eqref{6eq: w = w rho}, and \eqref{6eq: w = w rho, 1}.

For  $\sqrt{\rho} = \sqrt{|d_F| C} T/N$ with ${X T^{\vepsilon}} / {T} \Lt \sqrt{\rho} \Lt 1$,    or $\rho = 0$ in the $+$ case,  by the formulae in  \eqref{5eq: explicit T(n; c)} and  \eqref{6eq: w = Phi} as in Lemmas \ref{lem: exp sum} and \ref{lem: w (u; Lambda)},  we   transform $	S_{\rho}^{\pm} (N, C)$ into the form: 
	\begin{align} \label{7eq: S (N, C), 2}
M T^2 \mathop{\sum_{\RN (c) \sim C} }_{16 | c }  \frac 1 {\sqrt{\RN(c)}}   \mathop{\mathop{\sum \sum}_{  \RN(n_1), \RN(n_2) \asymp N_{\pm} }  }_{ (n_1, c) = (n_2, c)   }      h_{\delta} (c_{\mathrm{o}}, d_{\mathrm{o} } )  \frac{  g_{\gamma  } (m_1m_2 \widebar{c}_{\mathrm{o} } ) } { \sqrt{\RN(m_1 m_2  )}   }	    \Big( \frac {m_1 m_2} {c^{\star}} \Big)   \Upsilon_{\rho}^\pm  \Big( \frac {n_1 n_2} {4 \updelta_F c} \Big),
\end{align}
where $N_+ = N$, $N_- = \rho^2 N $ (by \eqref{6eq: support in u}), and $c_{\mathrm{o} }$, $c^{\star}$,  $d$,  $d_{\mathrm{o} }$, $m_1$, $m_2$, $\delta$, $\gamma$ are the notations from Lemma \ref{lem: exp sum}. 

\begin{remark}\label{rem: 3rd Poisson} 
As mentioned before,  the length  $ N_{\pm} = N$ or $\rho^2 N $  of the $n_1$- and $n_2$-sums after double Poisson  is {\it not} reduced (if $\rho \asymp 1$ in the $-$ case). It may be checked that the third Poisson as in Khan--Young \cite{Khan-Young-Sym2} will neither reduce the length of the $c$-sum nor bring in any structural change.   Therefore, after some clean-ups,  we shall proceed directly to the quadratic large sieve of Heath-Brown. 
\end{remark}

\subsection{Removal of Mellin Integrals and  {\large $p_2$}-Power Sums} 

Next, in the sum \eqref{7eq: S (N, C), 2}, we insert the expansion of $ g_{\gamma} (n)$ and the Mellin-integral expressions for $ \Upsilon_{\rho}^{\pm} (z) $ as in \eqref{5eq: g (n)},  \eqref{6eq: Phi o}, \eqref{6eq: Phi}, and \eqref{6eq: inert Mellin},   and then split the triple summation by 
\begin{align*}
	c = p_2^{k} c_{\mathrm{o}}, \quad n_1 = p_2^{ {h_1}} n_{1 \mathrm{o}}, \quad   n_2 = p_2^{ {h_2}} n_{2 \mathrm{o}}, \qquad 2 \nmid \RN (c_{\mathrm{o}} n_{1 \mathrm{o}} n_{2 \mathrm{o}}) . 
\end{align*} 
Since  the third Poisson (or functional equation) as in Khan--Young \cite{Khan-Young-Sym2} is not applied here, the triple Mellin integrals and the $p_2$-power sums will be estimated in the trivial manner. We prefer not to exhibit the complicated resulting formula. Nevertheless, it not hard to see that one is reduced to the type of sums as follows. 

\begin{defn}
	 Let $C, N_1, N_2, V, U \Gt 1$. Let $(\bfnu, \bfkappa) \in \boldsymbol{\EuScript{A}}  (V, U)$ {\rm(}$\boldsymbol{\EuScript{A}}  (V, U) =  {\EuScript{A}}  (V, U) \allowbreak  \times \allowbreak {\EuScript{A}}  (V, U)$ as in Definition {\rm\ref{defn: rectangle}}{\rm)}. Let 
	 $\omega \in \widehat{(\SO / 8 \shskip \SO){}^{\times}}${\rm(}considered as Dirichlet character{\rm)}.   Define 
	 \begin{align}
	 	S^{\omega}_{\bfnu, \bfkappa}  (\bfN, C)     =    \sumo_{\RN (c) \sim C } \frac 1 {\sqrt{\RN(c)}}  \mathop{ {\sumo \sumo}   }_{(n_1, c) = (n_2, c)  }       h  (c, d)  \frac{\omega \vchi_{c^{\star}} (\bfm)  \vchi_{i \bfnu,  \bfkappa          } (\bfn )  }  {\sqrt{\RN (\bfm)}}  \varww \bigg(    \frac   {  |\bfn|    } {\sqrt{\bfN }   }    \bigg) , 
	 \end{align}
 for  \begin{align*}
 	  h (c, d) = \left\{ \begin{aligned}
 	  	& O (1) , & & \text{ if } c^{\star} d  \, |  \shskip  c   ,\\
 	  	& 0 , & & \text{ if otherwise},  
 	  \end{aligned} \right.
 \end{align*} 
 and  $\varww (\boldsymbol{x}) = \varww(x_1) \varww (x_2) \in C_c^{\infty } (\BR_+^2)$, where  $ (d) = (n_1, c) = (n_2, c)$, $ \bfn = d \bfm $, and  the superscript {\small $\mathrm{o}$} means summation on odd integers in $\SO${\rm;} in the multi-variable notation $$ \omega \vchi_{c^{\star}} (\bfm) = \omega \vchi_{c^{\star}} (m_1 m_2) , \qquad  \vchi_{i \bfnu,  \bfkappa          } (\bfn ) = \vchi_{i \varnu_1,  \vvarkappa_1          } (n_1 ) \vchi_{i \varnu_2,  \vvarkappa_2          } (n_2 ) , $$  
 $$\RN (\bfm) = \RN (m_1 m_2), \qquad \frac {|\bfn|} {\sqrt{\bfN} } = \bigg(\frac{|n_1|} {\sqrt{N_1}}, \frac {|n_2|} {\sqrt{N_2}} \bigg), \ $$  where $\vchi_{i \varnu,  \vvarkappa          }$ and $\vchi_{q}$ are given as in  {\rm\eqref{2eq: chi vk}} and {\rm \eqref{2eq: Jacobi}}. 
\end{defn}

Note that we have  absorbed a certain character of $c_{\mathrm{o}}$ into $h_{\delta} (c_{\mathrm{o}}, d_{\mathrm{o}}) $, removed the subscript  {\small $\mathrm{o}$} from $c_{\mathrm{o}}$, $n_{1 \mathrm{o}}$, $n_{2 \mathrm{o}}$, $d_{\mathrm{o}}$, and smoothed the $n_1$, $n_2$-sums for later analysis (as in the proof of Lemma \ref{lem: weighted Heath-Brown}).  

\begin{prop}\label{prop: odd}

 We have 
\begin{align}
	S^{\omega}_{\bfnu, \bfkappa}  ( \bfN , C) \Lt  \big( \sqrt{C} + V + U    \big) (CVU)^{\vepsilon}  ,
\end{align}
uniformly in $\bfN  $.

\end{prop}

\subsection{Proof of Proposition \ref{prop: initial reduction}} 
Given Proposition \ref{prop: odd},  we are now ready to prove Proposition  \ref{prop: initial reduction} (and consequently Theorem \ref{thm: mean-Lindelof}).     

Proposition  \ref{prop: odd} is applied with 
\begin{align*}
	& N_1   = \frac {N_{\pm}} {\RN (p_2)^{h_1}}, \quad N_2 = \frac {N_{\pm}} {\RN (p_2)^{h_2}}, \quad C \ra \frac {C} {\RN (p_2)^{k}},  \\
	 V = U = XT^{\vepsilon} & \quad   \text{ if } \rho = 0, \qquad  V =  \frac {T^{1+\vepsilon}} {M} +X  T^{\vepsilon}, \ U = \sqrt{\rho} T^{1+\vepsilon}, \quad \text{ if } \rho \neq 0,  
\end{align*}
but we may as well use the `$p_2$-free' uniform bound $ O  ( ( \sqrt{C} + V + U    ) T^{\vepsilon}  ) $.   
Note that the $(\varnu, \vvarkappa)$-integral as in  \eqref{6eq: Phi o} or \eqref{6eq: Phi} has area $O (VU) = O (X^2 T^{\vepsilon})$ or $O (\sqrt{\rho} T (T/M+X) T^{\vepsilon}  )$,  while 
the double $(\bfnu, \bfkappa)$-integral in \eqref{6eq: inert Mellin} has area  $O (X^4 T^{\vepsilon})$. 
As for the $k$, $h_1$, $h_2$-sums, 
 there are only $O (\log^3 T)$ many summation points.  

By 
the discussion above,  we have
\begin{align*}
	S^{+}_{0} (N, C) \Lt { M T^{2} } 	\cdot \frac{X^5 \sqrt{C}} {N}    \cdot X^6 \cdot    (\sqrt{C}+ X  )   T^{\vepsilon}  \Lt X^{12} M T^{2+\vepsilon}. 
\end{align*}

Similarly, $$S^{\pm}_{\rho} (N, C) \Lt { M T^{2} } 	\cdot  \frac {X^4} {MT}  \cdot    X^4\sqrt{\rho}    T  \bigg( \frac {T} { M } + X \bigg) \cdot  \bigg(\sqrt{C} + \sqrt{\rho} T  + \frac {T} { M } + X   \bigg) T^{\vepsilon}.  $$
As $     {\rho} \Lt 1$, it follows that
$$S^{\pm}_{\rho} (N, C) \Lt { X^8 T^2} 	\bigg( \frac {T} { M } + X \bigg)  (\sqrt{C} +    {T}      ) T^{\vepsilon} \Lt \frac {X^9 T^{4+\vepsilon}} {M} + X^{10} T^{3+\vepsilon}.  $$

\subsection{Proof of Proposition \ref{prop: odd}} 

Finally, we prove Proposition \ref{prop: odd} by the following corollary of the quadratic large sieve of Heath-Brown as in Lemma \ref{lem:Heath-Brown LS}. 

\begin{lem}\label{lem: weighted Heath-Brown}
	 Let  $\omega \in \widehat{(\SO / 8 \shskip \SO){}^{\times}} $   and $ \varww \in C_c^{\infty} (\BR_+) $.  Then 
	 \begin{align}\label{7eq: weighted Heath-Brown} 
	 	\sumf_{ \RN(q) \sim Q  }   \frac 1 {\sqrt{\RN (q)}} 
	 	\bigg|   \sum_{n } \frac{\omega \vchi_{q}   \vchi_{i \varnu, \vvarkappa } (n)}{\sqrt{\RN(n)}  }  \varww \bigg( \hskip -1pt  \frac   {  |n|    } {\sqrt{N  } }   \hskip -1pt  \bigg)  \bigg|^2   \Lt     \big(\hskip -1pt  \sqrt{Q} + \sqrt{1 + \varnu^2 + \vvarkappa^2}   \big) \big(Q \big(1 + \varnu^2 + \vvarkappa^2\big) \big)^{\vepsilon} ,
	 \end{align}
 for any value of $N$. 
\end{lem}

\begin{proof}
	 As a consequence of Lemma  \ref{lem:Heath-Brown LS}, it follows from the approximate functional equation that 
	 \begin{align}\label{7eq: Heath-Brown, L}
	 	\sumf_{ \RN(q) \sim Q  }  \frac { \big|  L \big(s, \omega \vchi_{q}  \vchi_{i \varnu, \vvarkappa } \big) \big|^2 } {\sqrt{\RN (q)}} 
	 	 \Lt  \big(\hskip -1pt  \sqrt{Q} + \sqrt{|s - i \varnu|^2 + \vvarkappa^2}   \big) \big(Q \big(|s - i \varnu|^2 + \vvarkappa^2\big)\big)^{\vepsilon} ,
	 \end{align}
 for $\mathrm{Re} (s) = 1/2$. This is the analogue of \cite[(3.3)]{Khan-Young-Sym2}. See also \cite[Corollary 2]{Heath-Brown-Quad-LS}. 
 
 By the Mellin inversion on $\BR_+$, we have 
 \begin{align}\label{7eq: Mellin}
 	\sum_{n } \frac{\omega \vchi_{q}   \vchi_{i \varnu, \vvarkappa } (n)}{\sqrt{\RN(n)}  }  \varww \bigg( \hskip -1pt  \frac   {  |n|    } {\sqrt{N  } }   \hskip -1pt  \bigg) = \frac 1 { \pi} \int_{-\infty}^{\infty}     L \big(1/2+ it, \omega \vchi_{q}   \vchi_{i \varnu, \vvarkappa } \big) \widetilde{\varww} (2i t) N^{i t} \nd t,
 \end{align}
where $  \widetilde{\varww} (s)$ is the usual Mellin transform of $\varww (x)$.  It is well-known that $ \widetilde{\varww} (s)$ is of rapid decay on vertical lines, so after inserting \eqref{7eq: Mellin} into the left-hand side of \eqref{7eq: weighted Heath-Brown}, we deduce \eqref{7eq: weighted Heath-Brown} directly from \eqref{7eq: Heath-Brown, L} with an application of Cauchy--Schwarz. 
\end{proof}

Let us write   $ (c) = q \mathfrak{h}^2  $ so as to avoid
 the    $\star$ and $\square$ in $c = c^{\star} c_{\square}^2 $. Let us also turn the $\bfn$-sum into the $\bfm$-sum by $\bfn = d \bfm$. By the AM--GM inequality, 
 \begin{align*}
 	 S^{\omega}_{\bfnu, \bfkappa}  (\bfN, C) \Lt R  ^{\shskip\omega}_{\varnu_1, \vvarkappa_1}  (N_1, C) + R^{\shskip \omega}_{\varnu_2, \vvarkappa_2}  (N_2, C), 
 \end{align*} 
where  $R^{\shskip\omega}_{\varnu, \vvarkappa }  (N , C)$ (for $(\varnu, \vvarkappa) \in {\EuScript{A}}  (V, U)$) is defined as the sum: 
\begin{align*}
	\begin{aligned}
	\sumo_{\RN (\mathfrak{h} ) \Lt \sqrt C } \frac 1 { \RN (\mathfrak{h}) }  \sum_{\frd | \mathfrak{h}^2 }  \,
	   \sumf_{\RN (q) \sim C / \RN (\mathfrak{h} )^2 } \frac 1 {\sqrt{\RN (q)}}	\bigg|   \sum_{ (m, \frd^{-1} \mathfrak{h}^2) = (1) } \frac{\omega \vchi_{q}   \vchi_{i \varnu, \vvarkappa } (m)}{\sqrt{\RN(m)}  }  \varww \bigg( \hskip -1pt  \frac   {  |m|    } {\sqrt{N /\RN (\frd) } }   \hskip -1pt  \bigg)  \bigg|^2 . 
	\end{aligned}
\end{align*} 
The coprimality  condition $ (m, \frd^{-1} \mathfrak{h}^2) = (1) $ is harmless as it may be easily handled by   M\"obius  inversion and Cauchy--Schwarz inequality. At any rate, it follows from Lemma \ref{lem: weighted Heath-Brown} that 
\begin{align*}
	R^{\shskip\omega}_{\varnu, \vvarkappa }  (N , C) & \Lt (CVU)^{\vepsilon}  \sumo_{\RN (\mathfrak{h} ) \Lt \sqrt C } \frac 1 { \RN (\mathfrak{h}) }  \sum_{\frd | \mathfrak{h}^2 } \sum_{\mathfrak{f} | \frd^{-1} \mathfrak{h}^2 } \mu (\mathfrak{f})^2  \bigg( \frac {\sqrt{C}} {\RN (\mathfrak{h})} + V + U \bigg) \\
	& \Lt \big( \sqrt{C} + V + U    \big) (CVU)^{\vepsilon},   
\end{align*}
as desired. 

\appendix

\section{Local Calculations for the Exponential Sum}\label{app: local exp sum}

By twisted multiplicativity \eqref{5eq: twisted mult} and \eqref{5eq: twist by units}, it suffices to study the exponential sum $  T (n_1, n_2; c)$ in  the prime-power case.  In this appendix, let $p$ be a prime 
 and assume $$c = p^{k}. $$   
In view of \eqref{5eq: (n,c)}, let $ (n_1, p^k) = (n_2, p^k) = (p^{l}) $ and write
\begin{align*}
	m_1 = n_1/ p^{l}, \qquad m_2 = n_2/ p^l.   
\end{align*}
Recall from \eqref{5eq: T, 1} that
\begin{align*}
	T (n_1, n_2; c) = \RN (c) \mathop{\sum}_{\valpha \shskip \in \SO/c \shskip \SO} \mathop{\sum_{\beta \shskip  \in (\SO /c \shskip\SO)^\times }}_{ n_2 \beta \equiv n_1 (\mathrm{mod}\, c)  } e_F \bigg[ \frac{  \valpha^2 \beta  - n_1 \valpha  }{c} \bigg]. 
\end{align*}
By the changes  $ \valpha \ra m_2 \valpha $ and $\beta \ra \overline{m}_2^{2} \beta$ (in the case   $l < k$), we may write
\begin{align}
	T (n_1, n_2; p^k) = T_{k   l} (m_1 m_2; p) ,
\end{align}
with
\begin{align}\label{app: T(m), 1}
	T_{k  l} (n ; p) =	\RN (p)^{k}  \sum_{\valpha \shskip \in \SO/ p^k \shskip \SO}    \mathop{\sum_{\beta \shskip  \in (\SO / p^k  \SO)^\times }}_{  \beta \equiv n (\mathrm{mod}\, p^{k-l} )  } e_F \bigg[ \frac{  \valpha^2  \beta   }{p^{k}} - \frac {n \valpha} {p^{k-l} }  \bigg], \qquad     (n, p^{k-l}) = (1). 
\end{align}
The case $l = k$ is simple (as the inner sum is the Ramanujan sum $S (\valpha^2, 0; p^k)$):
\begin{align}\label{app: case l=k}
	T_{k  k} (n; p) =	\left\{ \begin{aligned}
		&  \RN (p)^{5k/2} \big(1- \RN (p)^{-1}\big) , & &   k \text{ even,}  \\
		& 0, & & k \text{ odd.}
	\end{aligned} \right. 
\end{align} 
For $l < k$, the inner $\beta$-sum in \eqref{app: T(m), 1} yields the congruence condition $\valpha^2 \equiv 0 \, (\mathrm{mod}\, p^l) $, and hence 
\begin{align}\label{eq: T(m) 2}
	T_{k  l} (n ; p) =	\RN (p)^{k+l}  \mathop{\sum_{ \valpha \shskip \in \SO/ p^k   \SO} }_{\valpha^2 \equiv \shskip 0 \shskip (\mathrm{mod} \, p^l) }  e_F \bigg[ \frac{ n \valpha^2     }{p^{k}} - \frac {n \valpha} {p^{k-l} } \bigg], \qquad (n, p) = (1).   
\end{align}
Set $j =  \lceil l/2 \rceil $ (in other words, $l = 2j$ or $2j-1$). By   $ \valpha \ra p^j \valpha$, we obtain
\begin{align}\label{eq: T(m) 3}
	T_{k  l} (n ;p ) =	\RN (p)^{k+l+j}  \sum_{ \valpha \shskip \in \SO/ p^{k-2j}   \SO}  e_F \bigg[ \frac{ n (\valpha^2  - p^{l-j} \valpha )  }{p^{k-2j }}   \bigg].  
\end{align}

\begin{lem}\label{lem: odd}
	Let $2 \nmid \RN (p)$. Then  for  $l < k$ we have
	\begin{align}\label{app: T(m), k odd}
		T_{k  l} (n; p) = \epsilon (p)   \RN (p)^{3k/2 + l} \Big(\frac {n} {p} \Big) e_F \bigg[ -\frac {\widebar{4} n p^{2l} } {p^k} \bigg]   , 
	\end{align}
if $k $ is odd,  with 
\begin{align}\label{app: epsilon(p)}
	\epsilon (p)  = 
\left\{ \begin{aligned}
	& 1,  & & \text{ if }   \RN (p) \equiv 1  \, (\mathrm{mod}\, 4),  \\
& - i \vchi_-(p), & & \text{ if }   \RN (p) \equiv -1 \, (\mathrm{mod}\, 4),  
\end{aligned} \right. 
\end{align}
where  $\vchi_- : (\SO / 4 \SO)^{\times} \ra \{ \pm 1 \}$ is a character, and 
\begin{align}\label{app: T(m), k even}
	T_{k  l} (n; p) =  \RN (p)^{3k/2 + l} e_F \bigg[ -\frac {\widebar{4} n p^{2l} } {p^k} \bigg]   , 
\end{align}
if $k $ is even. 

\end{lem}

Note that $ \epsilon (p) = 1 $ in the case $ d_F = -4 $ as $ \RN (p) \equiv  1 \, (\mathrm{mod}\, 4)$ for every odd $p$.  For $d_F \neq - 4$, the quadratic character $\vchi_- (n)$ is determined by \begin{align}\label{app: chi-}
	  \vchi_- (n) \equiv \Tr (n \mu^2 / \sqrt{d_F}) \, (\mathrm{mod}\, 4) , 
\end{align} for $\RN (n) \equiv -1 \, (\mathrm{mod}\, 4)$  and for any integer $\mu \in \SO $ such that the right-hand side is odd. For more details, the reader is referred to \cite[Lemma 5]{BS-Gauss-Sums}.

\begin{lem}\label{lem: even}
	Let $p = 2$, $1+i$, or $\sqrt{2} i$ be even and non-split.  
	Assume $k \geqslant  
	v_p (16)$.  Set $j =  \lceil l/2 \rceil $. 
	Define   $\gamma = k -2 j \, (\mathrm{mod}\, 4)$.  Then   for  $l < k$ we have $	T_{k l} (n; p) = 0$ unless 
	\begin{align}\label{app: l > 2}
		l \geqslant v_p (4), 
	\end{align}
in which case 
	\begin{align}\label{app: T(n;p) even}
		T_{k l} (n; p) = \RN(p)^{3k/2 + l} g_{\gamma} (n)  e_F \bigg[ \hskip -1pt   - \frac{n p^{2l} }{4p^k} \bigg]  , 
	\end{align}
 where  
 \begin{align}\label{app: defn g (n)}
 	 g_{\gamma} (n) = g   \bigg( \frac{  n}  {p^{k - 2j}} \bigg) 
 \end{align}  
is a function on $(\SO/ 8 \shskip \SO)^{\times}$, dependent only on $\gamma$ if $ k - 2j \geqslant 4 $, and explicitly given in  {\rm\eqref{app: g(n/p)}}--{\rm\eqref{app: last g(n/pk)}}.  
\end{lem}

\subsection{Explicit Formulae for Gauss Sums}

For $n,\, c\in \SO$ with $(n,c) = (1)$, define the normalized Gauss sum 
\begin{equation}\label{app: Gauss}
	g \Big( \frac{n}{c} \Big)   = \frac 1 {\sqrt{\RN (c)}}  \sum_{\valpha \in \SO/c \shskip \SO} e_F \bigg[ \frac{n \valpha^2}{c} \bigg].
\end{equation} 
Note that $g (n/c)$ is twisted multiplicative: 
\begin{equation}\label{app: twisted mult}
	g \Big( \frac{n}{c_1 c_2} \Big) =  g \Big( \frac{\widebar{c}_2 n}{c_1 } \Big) g \Big( \frac{\widebar{c}_1 n}{c_2 } \Big), 
\end{equation}
for $(c_1,c_2) = 1$. 
Thus it suffices to consider  as in \cite{BS-Gauss-Sums}  the case when $\RN (c)$ is odd and the case $c = p^k$ for $p$ even and prime. 

For $\RN (c)$ odd, it follows from \cite[(5)]{BS-Gauss-Sums} that 
\begin{align}\label{app: G(n/c), c odd}
	  g \Big( \frac{n}{c} \Big)  =   \epsilon (c)   \Big( \frac{n}{c} \Big), 
\end{align}
with 
\begin{align}\label{app: epsilon(c)} 
	\epsilon (c)  = 
	\left\{ \begin{aligned}
		& 1,  & & \text{ if }   \RN (c) \equiv 1  \, (\mathrm{mod}\, 4),  \\
		& - i \vchi_-(c), & & \text{ if }   \RN (c) \equiv -1 \, (\mathrm{mod}\, 4).  
	\end{aligned} \right. 
\end{align} 
Note that their formula is simplified as the norm is always positive in the case of imaginary quadratic fields. 

For  $\RN (p) = 2$ or $4$,  the explicit formula for $g (n/p^k)$ is given in \cite[Lemma 3]{BS-Gauss-Sums}. 
Let us assume here that $p$ is non-split. For our purpose,  it suffices to know that $  g (n/p^k)  $ depends only on  $n \, (\mathrm{mod}\, 8 \shskip \SO)$ and 
$ \gamma = k \, (\mathrm{mod}\, 4)$, provided that $k  \geqslant 4$.     

Henceforth, let $k \geqslant 2$ 
as we always have 
\begin{align}\label{app: g(n/p)}
g(n/1) = 1,  \qquad 	g (n/p ) = 0 .   
\end{align}


For the simplest inert case $p = 2$ ($d_F \equiv 5 \, (\mathrm{mod} \, 8)$),  \cite[Lemma 3]{BS-Gauss-Sums} actually implies that  $  g (n/2^k)  $ depends only on   $n \, (\mathrm{mod}\, 8 \shskip \SO)$ and   $ \delta = k \, (\mathrm{mod}\, 2)$:   
\begin{align}\label{app: g(n), inert}
	g (n/ 2^k) = \left\{ \begin{aligned}
		&2  \cdot  (2| \RN(n) )^{\delta} , & & \text{ if } \RN (n) \equiv 1 \, (\mathrm{mod}\, 4),  \\
		& 2i  \vchi_- (n) \cdot  (2| \RN(n) )^{\delta}  , & & \text{ if } \RN (n) \equiv -1  \, (\mathrm{mod}\, 4), 
	\end{aligned}
	\right.
\end{align} 
where $ ( 2 | \, \cdot \shskip )$ is the rational (Kronecker--)Jacobi symbol and $\vchi_- $ is the quadratic character on $(\SO / 4 \shskip \SO)^{\times}$ defined by \eqref{app: chi-}.   Note that in their notation  $ M = 2^k $, $ \Delta = 4 \RN (n) $, $\Delta_2 = \pm 4$ for $\RN (n) \equiv \pm 1 \, (\mathrm{mod}\, 4)$, and clearly $\Delta_2 | M$. Moreover, in view of \eqref{app: chi-}, their  $(-1|A) = \vchi_- (n)$ in the case $ \RN (n) \equiv -1  \, (\mathrm{mod}\, 4) $.  


\delete{Let us introduce $\delta \equiv k \, (\mathrm{mod}\, 2)$ and, for $d_F = -4,    -8$,  the quadratic character $\vchi_{\flat} (n)$ on $(\SO / 8 \shskip \SO)^{\times}$ 
determined by  \begin{align*}
	 \vchi_{\flat} (n) = \left\{ \begin{aligned}
	 	& 1,  & & \text{ if }   \Tr (n \mu^2 / \sqrt{d_F}) \equiv \pm 1 \, (\mathrm{mod}\, 8), \\
	 & - 1,  & & \text{ if }   \Tr (n \mu^2 / \sqrt{d_F}) \equiv \pm 3 \, (\mathrm{mod}\, 8), 
	 \end{aligned}
	 \right. 
\end{align*}  
for $\RN (n) $ odd  and for any integer $\mu \in \SO $ such that trace is odd. }

For the case that $p $ is ramified,  in their notation $ M = 2^{\lfloor k/2 \rfloor +\delta}$,  $ \Delta = 2^{2+\delta} \RN (n) $, $\Delta_2 = \pm 2^{2+\delta}$ for $\RN (n) \equiv \pm 1 \, (\mathrm{mod}\, 4)$, and moreover
\begin{align*}
	 a \equiv \left\{\begin{aligned}
	 	&\Tr (   n / 2 \sqrt{2} i) \, (\mathrm{mod}\, 2), & & \text{ if } \delta = 0, \\
	 	&\Tr (   n / 2  )\, (\mathrm{mod}\, 2), & & \text{ if } \delta = 1,  
	 \end{aligned}
	 \right.  
\end{align*}  if $p = \sqrt{2} i$ ($d_F = -8$),  
\begin{align*}
	a \equiv \left\{\begin{aligned}
		  &\Tr ( n / 2  i ) \, (\mathrm{mod}\, 2), & & \text{ if } \gamma = 0, \\
		&    \Tr ( n / 2  ) \, (\mathrm{mod}\, 2), & & \text{ if } \gamma = 2, \\
		&\Tr ( n / 2  i ) + \Tr ( n / 2  )  \, (\mathrm{mod}\, 2), & & \text{ if } \delta = 1,  
	\end{aligned}
	 \right.
\end{align*}
if   $p = 1+i$ ($d_F = -4$). 
Note that  if $\delta = 1$ then $a$  must be odd, and  one may choose $A = a$ in this case.   To this end, for $k$ odd ($\gamma = 1, 3$) let
\begin{align*}
	A   = (-1)^{(1+\gamma)/2}  \Tr (n/2), \qquad A   =  \Tr \big(i^{(1-\gamma)/2}  (1+i) n/2 \big), 
\end{align*}
if $p = \sqrt{2}i$ or $1+i$, and write
\begin{align}
	\tau_{\gamma} (n) = \left\{ \begin{aligned}
		& 1  , & & \text{ if }  \gamma = 0, 2, \\
		&(2|A)  , & &   \text{ if } \gamma = 1, 3. 
	\end{aligned}
	\right.   
\end{align} 
Note that $\tau_{\gamma} : (\SO / 8 \shskip \SO)^{\times} \ra \{ \pm 1 \}$. 

It follows from  \cite[Lemma 3]{BS-Gauss-Sums} that
\begin{align}
g (n/p^2) =  \left\{ \begin{aligned}
	&2  , & &   \text{ if } \Tr (n / \sqrt{d_F})  \text{ is even (or odd)},  \\
	& 0  , & & \text{ if }  \Tr (n / \sqrt{d_F})   \text{ is odd  (or even)},
\end{aligned}
\right.  
\end{align}
for ${d_F} = - 8$ (or $d_F = -4$), 
\begin{align}
	g (n/p^3) = 0, 
\end{align}
for $k \geqslant 4$ that
\begin{align}
	g (n/ p^k) =  \left\{ \begin{aligned}
		&2    \tau_{\gamma} (n)   , & & \text{ if } \RN (n) \equiv 1 \, (\mathrm{mod}\, 8),  \\
		& 2i \vchi_- (n) \cdot \tau_{\gamma} (n)  (-1)^{(\gamma-\delta)/ 2}   , & & \text{ if } \RN (n) \equiv 3  \, (\mathrm{mod}\, 8) ,
	\end{aligned}
	\right.
\end{align}   
if  $p = \sqrt{2} i$, and  
\begin{align}\label{app: last g(n/pk)}
	g (n/ p^k) = \left\{ \begin{aligned}
		& - 2   , & &\text{ if } \delta = 0 \text{ and } \mathrm{Tr} (n / 2 i) \equiv 2 \, (\mathrm{mod}\, 4),  \\
		&2  \tau_{\gamma} (n)  (2| \RN(n) )^{(\gamma-\delta)/2}  , & &  \text{ if otherwise},
	\end{aligned}
	\right.
\end{align}  
if  $p = 1+i$.
Recall that in the case $d_F = -4$ we always have $\RN (n) \equiv 1 \, (\mathrm{mod}\, 4)$ if $\RN (n)$ is odd.

\subsection{Proof of Lemma \ref{lem: odd}}

By completing the square, \eqref{eq: T(m) 3} is rewritten as
\begin{align*}
	 T_{k  l} (n;p ) =	\RN (p)^{3k/2+l } e_F \bigg[ \hskip -1pt    -\frac {\widebar{4} n p^{2l} } {p^k} \bigg]  g \bigg(\frac {n} {p^{k-2j}} \bigg), 
\end{align*}
by \eqref{app: Gauss}. Then \eqref{app: T(m), k odd} and \eqref{app: T(m), k even} follow directly from \eqref{app: G(n/c), c odd}.

\subsection{Proof of Lemma \ref{lem: even}}
First, in the inert case $p = 2$, we need to prove $ T_{kl}(n ; 2) = 0 $ for $ k \geqslant 4 > 2 l$. This may be seen if we let  $\valpha = 2^{k-2j-1} u + v$ in \eqref{eq: T(m) 3} so that $ \valpha^2 \equiv v^2 \, (\mathrm{mod}\, 2^{k-2j})$, then the $u$-sum vanishes as $l-j = 0$.  Similarly, in the ramified case $p = 1+i$ or $ \sqrt{2}i$, we may prove $ T_{kl}(n ; p) = 0  $ for $k \geqslant 8 > 2 l $, by letting $\valpha = p^{k-2j-2} u + v$  in \eqref{eq: T(m) 3}. 
 
Given $l \geqslant v_p (4)$ so that  $2 | p^{l - j}$,   \eqref{eq: T(m) 3} is turned into \eqref{app: T(n;p) even}  by  $ \valpha^2  - p^{l-j} \valpha = (\valpha - p^{l-j}/2)^2 - p^{2l-2j}/ 4 $.

\subsection{Proof of Lemma \ref{lem: exp sum}} 

By twisted multiplicativity  \eqref{5eq: twisted mult} and \eqref{5eq: twist by units},  Lemma \ref{lem: exp sum} is a direct consequence of    \eqref{app: case l=k} and  Lemmas  \ref{lem: odd},  \ref{lem: even}, along with  \eqref{app: twisted mult} and \eqref{app: G(n/c), c odd}.  
According to   \eqref{app: case l=k}, we need to add (alter) the definition of $g_{0} (n) $ in the case $k = l$: 
\begin{align}\label{app: defn g0(n)}
	g_{0} (n) = 	\left\{ \begin{aligned}
		&   1- 1/\RN (p_2)   , & &   k \text{ even,}  \\
		& 0, & & k \text{ odd.}
	\end{aligned} \right. 
\end{align}  

\begin{remark}\label{rem: 4 | n1n2}
The purpose of  $ \SB \hookrightarrow \SB (16)$ introduced at the beginning is to ensure $ l \geqslant v_{p_2} (4)$ as in  \eqref{app: l > 2}, or equivalently  $ 4 | (n_1, n_2)$ as in \eqref{5eq: non-vanishing}.  It has been used in the proof of Lemma \ref{lem: even}. Also, note that Lemma  \ref{lem: odd} yields $ e_F [- \widebar{4 } \widebar{p}_2^k  n_1 n_2/ c_{\mathrm{o}}]$, so  $ 4 | (n_1, n_2)$ helps us  rewrite  it  as $   e_F [-   \widebar{p}_2^k  n_1 n_2/ 4 c_{\mathrm{o}}]$.  
\end{remark}

\section{Supplement: The Even Split Case}\label{app: local exp sum, split}

For completeness, in this appendix, we record some results in the even split case  
\begin{align*}
	c = p^{k} q^{r}, \qquad p = \frac {1 + \sqrt{7}i}  2, \ q = \frac { 1 - \sqrt{7}i}  2.  
\end{align*}   
It is similar to the inert and ramified cases  but requires more notation, so the reader who is content with the simpler cases may skip this appendix.    

Let $(n_1,p^k q^r) = (n_2,p^k q^r) = (p^l q^s)$ and $n_1 = p^l q^s m_1$, $n_2 = p^l q^s m_2$.
Then with the definition in \eqref{app: T(m), 1} we have
\begin{align}
	T (n_1, n_2; p^k q^r) = T_{k  l} (\widebar{q}^{r} q^{2s} m_1 m_2 ;p) T_{r s} (\widebar{p}^{k} p^{2l} m_1 m_2 ;q).
\end{align}
For $k = l$ (or $r = s$), $ T_{k  k} (n ;p) $ (or $T_{rr} (n; q)$) has been evaluated in \eqref{app: case l=k}.

\begin{lem}\label{lem: split}
	 Assume $ k \geqslant 4$. Set $j = \lceil l/2\rceil$. Let $h$ be an integer. Define $\delta = k \allowbreak\,(\mathrm{mod}\, 2)$ and $\epsilon = h  \,(\mathrm{mod}\, 2)$.   Then   for $l < k$  we have  $T_{k  l} (q^h n ;p) = 0$ unless $l \geq 2$, in which case
	 \begin{align}\label{Beq: even split}
	 	T_{k  l} (q^h n ;p) = \RN(p)^{3k/2 + l} g_{\delta,\epsilon}(n; p ) e_F \bigg[ \hskip -1.5pt      - \frac{ q^{h} p^{2l} n  }{4p^k} \bigg] ,
	 \end{align}
 where  
 \begin{align}
 	 g_{\delta,\epsilon}(n ; p) = g \bigg( \frac{q^h n}  {p^{k - 2j}} \bigg) 
 \end{align}  is a function on $ (\SO / 8\SO)^\times $, dependent only on $\delta$ and $\epsilon$ if $k-2j \geqslant 2$, and explicitly given in {\rm \eqref{app2: g (n/1), g(n/p)}}--{\rm\eqref{app2: A(n)}}{\rm;} it is understood that $ q^h = \widebar{q}^{\, -h} $ if $h < 0$. 
\end{lem}

Lemma \ref{lem: split} is still valid if the roles of $p$ and $q$ are exchanged. 

Consequently, Lemma \ref{lem: exp sum} may be readily extended to the split case with $h_{\delta} (c_{\mathrm{o}}, d_{\mathrm{o} } )$  and  $ g_{\gamma  } (n) $ replaced by some $h_{\delta, \epsilon} (c_{\mathrm{o}}, d_{\mathrm{o} } )$  and  $ g_{\delta, \epsilon  } (n) $ for $ \delta = k \, (\mathrm{mod}\, 2)$  and $ \epsilon = r \, (\mathrm{mod}\, 2) $.

\subsection{Explicit Formulae for Gauss Sums: The Even Split Case}

For $\RN (n)$ odd, consider the Gauss sum $ g (q^{h} n / p^k) $; the results for $g (p^{ f} n / q^{r}) $ are similar by symmetry. For  $h < 0$  one should view $q^{h} = \widebar{q}^{\, - h}$. However,  by the definition in \eqref{app: G(n/c), c odd} it is clear that $ g (q^{ h} n / p^k) = g (q^{ -h} n / p^k) $. 

Still we have
\begin{align}\label{app2: g (n/1), g(n/p)}
g (q^{ h} n / 1 ) = 1, 	\qquad   g (q^{ h} n /p ) = 0 .   
\end{align} 
For $k \geqslant 2$, it turns out that $  g (q^{ h} n/p^k)  $ depends only on   $n \, (\mathrm{mod}\, 8 \shskip \SO)$, 
$ \delta = k \, (\mathrm{mod}\, 2)$, and $ \epsilon = h \, (\mathrm{mod}\, 2) $.   
 Indeed, by \cite[Lemma 3]{BS-Gauss-Sums} we have
\begin{align}
	g ( q^{ h} n / p^k) = 
 	\sqrt{2} \cdot \big(2| A  (q^{  \epsilon+\delta} n  ) \big)^{1+\delta}  e \big( A  (q^{  \epsilon+\delta} n   )/8 \big),
\end{align}
with 
\begin{align}\label{app2: A(n)}
	 A (n) = \Tr (q^2 n /\sqrt{d_F}), \qquad A (q n) = \Tr (q n / \sqrt{d_F}) ,
\end{align}  
which are easily proven to be odd. 
For $h \geqslant 0$, in their notation  $M = 2^k$, $M \omega =  q^{k+h} n$,  $ \Delta =  2^{2+k + h }\allowbreak  \RN (n)$, and clearly $ 4 M | \Delta$ (so the arithmetic is quite different from the inert or ramified case). Note that their $A$ is replaced by our $A (n)$ or $A(qn)$ according as $k+h$ is even or odd (due to $A \equiv A( q^{  \epsilon+\delta} n  )\, (\mathrm{mod}\, 2^{2+\delta})$).   



\begin{thebibliography}{} 
		
		\bibitem[AHLQ]{AHLQ-Bessel}
		K.~Aggarwal, R.~Holowinsky, Y.~Lin, and Z.~Qi.
		\newblock A {B}essel delta method and exponential sums for {${\rm GL}(2)$}.
		\newblock {\em Q. J. Math.}, 71(3):1143--1168, 2020.
		
		\bibitem[Bal]{Balkanova-Sym2-Maass}
		O.~Balkanova.
		\newblock The first moment of {M}aa\ss\ form symmetric square {$L$}-functions.
		\newblock {\em Ramanujan J.}, 55(2):761--781, 2021.
		
		\bibitem[BB]{Blomer-Brumley}
		V.~Blomer and F.~Brumley.
		\newblock On the {R}amanujan conjecture over number fields.
		\newblock {\em Ann. of Math. (2)}, 174(1):581--605, 2011.
		
		\bibitem[BBCL]{BBCL-PGT-Picard-3}
		A.~Balog, A.~Bir\'o, G.~Cherubini, and N.~Laaksonen.
		\newblock Bykovskii-type theorem for the {P}icard manifold.
		\newblock {\em Int. Math. Res. Not. IMRN}, (3):1893--1921, 2022.
		
		\bibitem[BCC{\etalchar{+}}]{BCCFL-PGT-Picard-1}
		O.~Balkanova, D.~Chatzakos, G.~Cherubini, D.~Frolenkov, and N.~Laaksonen.
		\newblock Prime geodesic theorem in the 3-dimensional hyperbolic space.
		\newblock {\em Trans. Amer. Math. Soc.}, 372(8):5355--5374, 2019.
		
		\bibitem[BF1]{BF-Mean-Sym2}
		O.~Balkanova and D.~Frolenkov.
		\newblock The mean value of symmetric square {$L$}-functions.
		\newblock {\em Algebra Number Theory}, 12(1):35--59, 2018.
		
		\bibitem[BF2]{BF-PGT-Picard-2}
		O.~Balkanova and D.~Frolenkov.
		\newblock Prime geodesic theorem for the {P}icard manifold.
		\newblock {\em Adv. Math.}, 375:107377, 42, 2020.
		
		\bibitem[BF3]{BF-Sym2-Non-vanishing}
		O.~Balkanova and D.~Frolenkov.
		\newblock Non-vanishing of {M}aass form symmetric square {$L$}-functions.
		\newblock {\em J. Math. Anal. Appl.}, 500(2):Paper No. 125148, 23, 2021.
		
		\bibitem[BF4]{BF-Picard-Sym2}
		O.~Balkanova and D.~Frolenkov.
		\newblock The second moment of symmetric square {$L$}-functions over {G}aussian
		integers.
		\newblock {\em Proc. Roy. Soc. Edinburgh Sect. A}, 152(1):54--80, 2022.
		
		\bibitem[BF5]{BF-Voronoi}
		O.~Balkanova and D.~Frolenkov.
		\newblock A {V}oronoi summation formula for non-holomorphic {M}aass forms of
		half-integral weight.
		\newblock {\em Monatsh. Math.}, 203(4):733--764, 2024.
		
		\bibitem[BF6]{BF-KY-Sym2-2}
		O.~Balkanova and D.~Frolenkov.
		\newblock Hybrid subconvexity for {M}aass form symmetric-square
		{$L$}-functions.
		\newblock {arXiv:2408.06735}, 2024.
		
		\bibitem[BKY]{BKY-Mass}
		V.~Blomer, R.~Khan, and M.~P. Young.
		\newblock Distribution of mass of holomorphic cusp forms.
		\newblock {\em Duke Math. J.}, 162(14):2609--2644, 2013.
		
		\bibitem[Blo1]{Blomer-Sym2-Hol}
		V.~Blomer.
		\newblock On the central value of symmetric square {$L$}-functions.
		\newblock {\em Math. Z.}, 260(4):755--777, 2008.
		
		\bibitem[Blo2]{Blomer-Hecke-Quad}
		V.~Blomer.
		\newblock Sums of {H}ecke eigenvalues over values of quadratic polynomials.
		\newblock {\em Int. Math. Res. Not. IMRN}, (16):Art. ID rnn059. 29, 2008.
		
		\bibitem[Blo3]{Blomer}
		V.~Blomer.
		\newblock Subconvexity for twisted {$L$}-functions on {${\rm GL}(3)$}.
		\newblock {\em Amer. J. Math.}, 134(5):1385--1421, 2012.
		
		\bibitem[BM]{BM-Kuz-Spherical}
		R.~W. Bruggeman and R.~J. Miatello.
		\newblock Sum formula for {${\rm SL}_2$} over a number field and {S}elberg type
		estimate for exceptional eigenvalues.
		\newblock {\em Geom. Funct. Anal.}, 8(4):627--655, 1998.
		
		\bibitem[BS]{BS-Gauss-Sums}
		H.~Boylan and N.-P. Skoruppa.
		\newblock Explicit formulas for {H}ecke {G}auss sums in quadratic number
		fields.
		\newblock {\em Abh. Math. Semin. Univ. Hambg.}, 80(2):213--226, 2010.
		
		\bibitem[Byk]{Bykovskii}
		V.~A. Bykovski\u{\i}.
		\newblock Density theorems and the mean value of arithmetic functions on short
		intervals.
		\newblock {\em Zap. Nauchn. Sem. S.-Peterburg. Otdel. Mat. Inst. Steklov.
			(POMI)}, 212:56--70, 196, 1994.
		
		\bibitem[Cai]{Cai-PGT}
		Y.~Cai.
		\newblock Prime geodesic theorem.
		\newblock {\em J. Th\'eor. Nombres Bordeaux}, 14(1):59--72, 2002.
		
		\bibitem[CI]{CI-Cubic}
		J.~B. Conrey and H.~Iwaniec.
		\newblock The cubic moment of central values of automorphic {$L$}-functions.
		\newblock {\em Ann. of Math. (2)}, 151(3):1175--1216, 2000.
		
		\bibitem[CWZ]{C-Wu-Z-KB-Formula}
		G.~Cherubini, H.~Wu, and G.~Z\'abr\'adi.
		\newblock On {K}uznetsov-{B}ykovskii's formula of counting prime geodesics.
		\newblock {\em Math. Z.}, 300(1):881--928, 2022.
		
		\bibitem[DK]{Khan-Das-3rd-Moment}
		S.~Das and R.~Khan.
		\newblock The third moment of symmetric square {$L$}-functions.
		\newblock {\em Q. J. Math.}, 69(3):\allowbreak1063--1087, 2018.
		
			\bibitem[EGM]{EGM}
		J.~Elstrodt, F.~Grunewald, and J.~Mennicke.
		\newblock {\em Groups {A}cting on {H}yperbolic {S}pace}.
		\newblock Springer Monographs in Mathematics. Springer-Verlag, Berlin, 1998.
		
		\bibitem[Fro]{BF-KY-Sym2-1}
		D.~Frolenkov.
		\newblock The second moment of {M}aass form symmetric square {$L$}-functions at
		the central point.
		\newblock {arXiv:2408.05929}, 2024.
		
		\bibitem[GJ]{GJ-GL(2)-GL(3)}
		S.~Gelbart and H.~Jacquet.
		\newblock A relation between automorphic representations of {${\rm GL}(2)$}\
		and {${\rm GL}(3)$}.
		\newblock {\em Ann. Sci. \'Ecole Norm. Sup. (4)}, 11(4):471--542, 1978.
		
		\bibitem[GL]{Heath-Brown-Quad-LS-3}
		L.~Goldmakher and B.~Louvel.
		\newblock A quadratic large sieve inequality over number fields.
		\newblock {\em Math. Proc. Cambridge Philos. Soc.}, 154(2):193--212, 2013.
		
		\bibitem[HB]{Heath-Brown-Quad-LS}
		D.~R. Heath-Brown.
		\newblock A mean value estimate for real character sums.
		\newblock {\em Acta Arith.}, 72(3):235--275, 1995.
		
		\bibitem[IK]{IK}
		H.~Iwaniec and E.~Kowalski.
		\newblock {\em Analytic {N}umber {T}heory},   {  American
			Mathematical Society Colloquium Publications, Vol. 53}.
		\newblock American Mathematical Society, Providence, RI, 2004.
		
		\bibitem[IM]{Iwaniec-Michel-Sym2}
		H.~Iwaniec and P.~Michel.
		\newblock The second moment of the symmetric square {$L$}-functions.
		\newblock {\em Ann. Acad. Sci. Fenn. Math.}, 26(2):465--482, 2001.
		
		\bibitem[Iwa]{Iwaniec-PGT}
		H.~Iwaniec.
		\newblock Prime geodesic theorem.
		\newblock {\em J. Reine Angew. Math.}, 349:136--159, 1984.
		
		\bibitem[Jac]{Jacquet-RTF}
		H.~Jacquet.
		\newblock On the nonvanishing of some {$L$}-functions.
		\newblock {\em Proc. Indian Acad. Sci. Math. Sci.}, 97(1-3):117--155 (1988),
		1987.
		
		\bibitem[Joh]{Faa-di-Bruno}
		W.~P. Johnson.
		\newblock The curious history of {F}a\`a di {B}runo's formula.
		\newblock {\em Amer. Math. Monthly}, 109(3):\allowbreak217--234, 2002.
		
		
		\bibitem[Kha]{Khan-Sym2-Non-vanishing}
		R.~Khan.
		\newblock Non-vanishing of the symmetric square {$L$}-function at the central
		point.
		\newblock {\em Proc. Lond. Math. Soc. (3)}, 100(3):736--762, 2010.
		
		\bibitem[Koy]{Koyama-PGT-Picard}
		S.-y. Koyama.
		\newblock Prime geodesic theorem for the {P}icard manifold under the
		mean-{L}indel\"of hypothesis.
		\newblock {\em Forum Math.}, 13(6):781--793, 2001.
		
		\bibitem[KPY]{KPY-Stationary-Phase}
		E.~M. Kiral, I.~Petrow, and M.w~P. Young.
		\newblock Oscillatory integrals with uniformity in parameters.
		\newblock {\em J. Th\'eor. Nombres Bordeaux}, 31(1):145--159, 2019.
		
		\bibitem[KY]{Khan-Young-Sym2}
		R.~Khan and M.~P. Young.
		\newblock Moments and hybrid subconvexity for symmetric-square {$L$}-functions.
		\newblock {\em J. Inst. Math. Jussieu}, 22(5):2029--2073, 2023.
		
		\bibitem[Lam]{Lam-Sym2-Hol}
		J.~W.~C. Lam.
		\newblock The second moment of the central values of the symmetric square
		{$L$}-functions.
		\newblock {\em Ramanujan J.}, 38(1):129--145, 2015.
		
		\bibitem[LG]{B-Mo2}
		H.~Lokvenec-Guleska.
		\newblock {\em {S}um {F}ormula for $\mathrm{SL}_2$ over {I}maginary {Q}uadratic
			{N}umber {F}ields}.
		\newblock Ph.D. Thesis. Utrecht University, 2004.
		
		\bibitem[Liu]{Liu-Sym2-1st-Moment}
		S.~Liu.
		\newblock The first moment of central values of symmetric square
		{$L$}-functions in the weight aspect.
		\newblock {\em Ramanujan J.}, 46(3):775--794, 2018.
		
		\bibitem[LQ1]{Qi-Liu-LLZ}
		S.-C. Liu. and Z.~Qi.
		\newblock Low-lying zeros of {$L$}-functions for {M}aass forms over imaginary
		quadratic fields.
		\newblock {\em Mathematika}, 66(3):777--805, 2020.
		
		\bibitem[LQ2]{Qi-Liu-Moments}
		S.-C. Liu and Z.~Qi.
		\newblock Moments of central {$L$}-values for {M}aass forms over imaginary
		quadratic fields.
		\newblock {\em Trans. Amer. Math. Soc.}, 375(5):3381--3410, 2022.
		
		\bibitem[LS]{Luo-Sarnak-QE}
		W.~Z. Luo and P.~Sarnak.
		\newblock Quantum ergodicity of eigenfunctions on {${\rm PSL}_2(\bold
			Z)\backslash \bold H^2$}.
		\newblock {\em Inst. Hautes \'Etudes Sci. Publ. Math.}, (81):207--237, 1995.
		
		\bibitem[Luo]{Luo-Sym2-2nd-Moment}
		W.~Luo.
		\newblock Central values of the symmetric square {$L$}-functions.
		\newblock {\em Proc. Amer. Math. Soc.}, 140(10):\allowbreak3313--3322, 2012.
		
		\bibitem[Mol]{Molteni-L(1)}
		G.~Molteni.
		\newblock Upper and lower bounds at {$s=1$} for certain {D}irichlet series with
		{E}uler product.
		\newblock {\em Duke Math. J.}, 111(1):133--158, 2002.
		
		\bibitem[Nel1]{Nelson-Sym2}
		P.~D. Nelson.
		\newblock Bounds for twisted symmetric square {$L$}-functions via half-integral
		weight periods.
		\newblock {\em Forum Math. Sigma}, 8:Paper No. e44, 21, 2020.
		
		\bibitem[Nel2]{Nelson-Eisenstein}
		P.~H. Nelson.
		\newblock Eisenstein series and the cubic moment for {$\text{PGL}_2$}.
		\newblock {arXiv:1911.06310}, 2019.
		
		\bibitem[Neu]{Neukirch-ANT}
		J.~Neukirch.
		\newblock {\em Algebraic {N}umber {T}heory},   {  Grundlehren der
			mathematischen Wissenschaften, Vol. 322}.
		\newblock Springer-Verlag, Berlin, 1999.
		
		\bibitem[Ono]{Heath-Brown-Quad-LS-2}
		K.~Onodera.
		\newblock Bound for the sum involving the {J}acobi symbol in {$\Bbb Z[i]$}.
		\newblock {\em Funct. Approx. Comment. Math.}, 41:71--103, 2009.
		
		\bibitem[Qi1]{Qi-Gauss}
		Z.~Qi.
		\newblock Subconvexity for twisted {$L$}-functions on {$\rm{GL}_3$} over the
		{G}aussian number field.
		\newblock {\em Trans. Amer. Math. Soc.}, 372(12):8897--8932, 2019.
		
		
		\bibitem[Qi2]{Qi-GL(3)}
		Z.~Qi.
		\newblock Subconvexity for {$L$}-functions on {$\rm GL_3$} over number fields.
		\newblock {\em J. Eur. Math. Soc. (JEMS)}, 26\allowbreak(3):\allowbreak1113--1192, 2024.
		
		
		\bibitem[Qi3]{Qi-Sym2-LS}
		Z.~Qi.
		\newblock {O}n the symmetric square large sieve inequality for {$\mathrm{PSL}_2
			(\mathbb{Z} {[i]}) \backslash \mathrm{PSL}_2 (\mathbb{C}) $} and the prime
		geodesic theorem for {$\mathrm{PSL}_2 (\mathbb{Z} {[i]}) \backslash
			\mathbb{H}^3 $}.
		\newblock {arXiv:2407.17959}, 2024.
		
		\bibitem[Sar]{Sarnak-H3}
		P.~Sarnak.
		\newblock The arithmetic and geometry of some hyperbolic three-manifolds.
		\newblock {\em Acta Math.}, 151(3-4):253--295, 1983.
		
		\bibitem[Shi]{Shimura-Sym2}
		G.~Shimura.
		\newblock On the holomorphy of certain {D}irichlet series.
		\newblock {\em Proc. London Math. Soc. (3)}, 31\allowbreak(1):79--98, 1975.
		
		\bibitem[SY]{Sound-Young-PGT}
		K.~Soundararajan and M.~P. Young.
		\newblock The prime geodesic theorem.
		\newblock {\em J. Reine Angew. Math.}, 676:\allowbreak105--120, 2013.
		
		\bibitem[TX]{Tang-Xu-Sym2-Maass}
		H.~Tang and Z.~Xu.
		\newblock Central value of the symmetric square {$L$}-functions related to
		{H}ecke-{M}aass forms.
		\newblock {\em Lith. Math. J.}, 56(2):251--267, 2016.
		
		\bibitem[Ven]{Venkatesh-BeyondEndoscopy}
		A.~Venkatesh.
		\newblock ``{B}eyond endoscopy'' and special forms on $\mathrm{GL}(2)$.
		\newblock {\em J. Reine Angew. Math.}, 577:23--80, 2004.
		
		\bibitem[Wei]{Weil-Unitary}
		A.~Weil.
		\newblock Sur certains groupes d'op\'{e}rateurs unitaires.
		\newblock {\em Acta Math.}, 111:143--211, 1964.
		
	\end{thebibliography}

 \newcommand{\etalchar}[1]{$^{#1}$}

\end{document}